\documentclass[preprint,12pt]{elsarticle}
\usepackage[T1]{fontenc}
\usepackage{newtxtext}

\usepackage{amsmath}

\usepackage{amsthm}
\usepackage{mathtools}
\usepackage{newtxmath}

\usepackage{bm}

\usepackage{booktabs}
\usepackage{multirow}
\usepackage{array}
\usepackage{adjustbox}
\usepackage{graphicx}
\usepackage{algorithm}
\usepackage{algorithmic}
\usepackage{xcolor}
\usepackage{setspace}
\usepackage{indentfirst}
\usepackage{float}
\usepackage{pdflscape}

\usepackage[section]{placeins}
\usepackage{siunitx}
\usepackage{microtype}
\usepackage{enumitem}
\usepackage[font=small,labelfont=bf]{caption}
\usepackage{hyperref}
\usepackage[nameinlink,capitalise,noabbrev]{cleveref}

\allowdisplaybreaks
\graphicspath{{figures/}}

\hypersetup{
  colorlinks=true,
  linkcolor=blue!55!black,
  citecolor=green!40!black,
  urlcolor=blue!65!black,
  filecolor=magenta,
  pdftitle={Hard-Trace First-Order Residual Learning for Coupled Stokes--Brinkman--Darcy Flow},
  pdfauthor={Zefeng Liu and Hongxing Rui}
}

\makeatletter
\def\ps@pprintTitle{%
  \let\@oddhead\@empty\let\@evenhead\@empty
  \def\@oddfoot{\hfil\thepage\hfil}%
  \let\@evenfoot\@oddfoot}
\makeatother
\biboptions{sort&compress}
\newtheorem{proposition}{Proposition}
\newtheorem{assumption}{Assumption}

\newcommand{\OS}{\Omega_{\mathrm S}}
\newcommand{\OB}{\Omega_{\mathrm B}}
\newcommand{\OD}{\Omega_{\mathrm D}}
\newcommand{\GSB}{\Gamma_{\mathrm{SB}}}
\newcommand{\GBD}{\Gamma_{\mathrm{BD}}}
\newcommand{\GextS}{\Gamma_{\mathrm S}^{\mathrm{ext}}}
\newcommand{\GextB}{\Gamma_{\mathrm B}^{\mathrm{ext}}}
\newcommand{\GextD}{\Gamma_{\mathrm D}^{\mathrm{ext}}}
\newcommand{\R}{\mathcal R}
\newcommand{\M}{\mathfrak M}
\newcommand{\MN}{\widehat{\mathfrak M}}

\hypersetup{hidelinks}
\begin{document}

\begin{frontmatter}

\title{Hard-Trace First-Order Residual Learning for Coupled
Stokes--Brinkman--Darcy Flow}

\author[1]{Zefeng Liu}
\ead{18552029283@163.com}

\author[1]{Hongxing Rui\corref{cor1}}
\ead{hxrui@sdu.edu.cn}
\cortext[cor1]{Corresponding author.}

\affiliation[1]{organization={School of Mathematics, Shandong University},
                city={Jinan},
                state={Shandong},
                postcode={250100},
                country={China}}

\begin{abstract}
\begin{singlespace}
Stokes--Brinkman--Darcy (SBD) models connect free flow and porous flow through a Brinkman layer of finite thickness. The proposed method for solving these coupled equations consists of a mixed first-order formulation, shared interface traces and hard traction constraints, and a pressure correction based on boundary mass balance. Independent stresses, Darcy flux and pressure-gradient auxiliaries avoid second derivatives of network outputs. Boundary liftings and shared traces impose exterior and interface kinematic conditions exactly, while a stress map enforces native Brinkman--Darcy force balance. An integral momentum relation determines the pressure correction after training, without reference-pressure labels. Across two manufactured solutions and three permeabilities, mean corrected upper-pressure relative $L^2$ errors range from 0.87\% to 2.58\%; correction reduces both upper-pressure $L^2$ errors and the full Brinkman-pressure $H^1$ error in every run. Paired controls assess the contributions of interface sharing and traction enforcement. A non-manufactured filtration problem tests mixed-boundary adaptation against an independently refined finite-element reference, with all four runs satisfying the prescribed field and physical-consistency criteria.
\end{singlespace}
\end{abstract}


\end{frontmatter}

\section{Introduction}
\label{sec:introduction}

Free-fluid flow coupled to a porous medium arises in filtration, groundwater transport and biological exchange. Stokes and Darcy equations describe the two regions, while interface laws transmit mass and force. The experiments of Beavers and Joseph \cite{BeaversJoseph1967}, Saffman's modification \cite{Saffman1971}, and the homogenization analysis of J\"ager and Mikeli\'c \cite{JagerMikelic2000} underpin classical tangential slip conditions. The sharp-interface Beavers--Joseph condition has a restricted range of validity, particularly for general filtration directions \cite{EggenweilerRybak2020Interface}. Brinkman's viscous-drag model \cite{Brinkman1949} allows a transition region to be resolved explicitly. Ruan and Rybak \cite{RuanRybak2026} derived and analyzed full- and hybrid-dimensional Stokes--Brinkman--Darcy (SBD) models. We consider their full-dimensional setting, in which a Brinkman layer of finite thickness connects the free-flow and Darcy regions.

Mixed finite-element methods provide established discretizations for coupled flow \cite{LaytonSchieweckYotov2002}. First-order system least squares introduces independent stresses or fluxes \cite{CaiManteuffelMcCormick1997,BochevGunzburger2009}, including for Stokes--Darcy coupling \cite{MunzenmaierStarke2011}. Neural counterparts include deep least-squares methods \cite{CaiChenLiuLiu2020,BersetcheBorthagaray2023} and mixed networks for heterogeneous domains \cite{RezaeiEtAl2022Mixed}. These formulations provide a way to train on first derivatives. Here we use a mixed state to represent the three regional equations, and assess the resulting neural solver against exact solutions and an independent finite-element reference.

Boundary-satisfying neural trial functions date to Lagaris et al. \cite{Lagaris1998}. Physics-informed neural networks (PINNs) evaluate equation residuals by automatic differentiation \cite{RaissiPerdikarisKarniadakis2019,KarniadakisEtAl2021Review}; NSFnets apply this approach to incompressible flow \cite{JinEtAl2021NSFnets}. In coupled Stokes--Darcy problems, Pu and Feng \cite{PuFeng2022} studied small-parameter and interface difficulties using weighting, parallel networks and adaptive activation functions. When interface equations remain penalty terms, their accuracy depends on the balance reached during optimization. Gradient imbalance and different convergence rates among loss components help explain this difficulty \cite{WangTengPerdikaris2021,WangYuPerdikaris2022NTK}.

The first issue in SBD coupling is that its two interfaces require different shared quantities. Domain-decomposition methods such as cPINNs and FBPINNs provide local neural representations \cite{JagtapKharazmiKarniadakis2020,MoseleyMarkhamNissenMeyer2023}. Distance functions and transfinite interpolation impose exterior data exactly \cite{SukumarSrivastava2022,BerroneCanutoPintoreSukumar2023}, and von Tresckow et al. \cite{TresckowIonLoukrezis2026} use dedicated interface networks to enforce solution continuity between geometric patches. For Stokes/Darcy flow, Lu, Zhang and Zhu \cite{LuZhangZhu2025} impose hard exterior conditions but retain soft interface equations. In the present SBD model, full velocity must be shared at Stokes--Brinkman (SB), whereas only normal flux is shared at Brinkman--Darcy (BD); Darcy tangential flux must remain independent. The trial fields must encode this distinction together with the BD traction law. Exact sharing alone does not impose that force balance.

The second issue is pressure accuracy under strong Brinkman drag. FO-PINNs combine first-order reduction with exact boundary constraints \cite{GladstoneEtAl2023FOPINN}, and pressure-equation augmentation has been studied for incompressible flow \cite{GorayaEtAl2022PressureAugmentation}. In our scaled SBD formulation, however, a compatible low-order pressure perturbation survives both exact kinematic constraints and pressure augmentation. Dividing Brinkman momentum by the large drag scale weakens the loss response to this perturbation. We derive its loss contribution and use an integral momentum relation, together with prescribed boundary mass balance, to determine a correction after training. The coefficient requires boundary data, forcing and learned stress, without reference-pressure labels.

We therefore construct shared velocity and flux traces for the two interfaces, impose native BD traction through a stress lifting, and correct the remaining low-order pressure direction. The mixed formulation avoids second derivatives of network outputs. Verification uses the method of manufactured solutions (MMS), with two cases denoted MMS1 and MMS2. Five-seed benchmarks compare the kinematic-hard-trace configuration with two soft baselines; paired controls assess BD traction and interface sharing. For the complete method, pressure correction reduces upper-pressure field and Brinkman-pressure derivative errors in every run across two manufactured solutions and three permeabilities. A boundary-driven filtration (BDF) test then checks mixed-boundary adaptation against an independently refined finite-element solution.

\Cref{sec:model,sec:method} present the model and construction. \Cref{sec:experiments} defines the comparison protocols, and \cref{sec:results} reports the manufactured tests, component controls and BDF application. \Cref{sec:conclusion} discusses the findings and the next steps toward three-dimensional flow.

\section{Mathematical model}
\label{sec:model}

We study the full-dimensional, three-rectangle configuration with constant isotropic coefficients, nonsymmetric pseudo-stress and zero Stokes--Brinkman traction jump.

\subsection{Geometry, boundaries, and orientation}

Let \(\Omega=(0,1)\times(0,2)\).  Its closure is partitioned into the
closures of three nonoverlapping rectangles,
\begin{equation}
\OD=(0,1)\times(0,0.9),\qquad
\OB=(0,1)\times(0.9,1.1),\qquad
\OS=(0,1)\times(1.1,2).
\label{eq:domains}
\end{equation}
The horizontal interfaces are $\GSB=\{(x,1.1):0<x<1\}$ and $\GBD=\{(x,0.9):0<x<1\}$. Exterior portions $\Gamma_r^{\mathrm{ext}}=\partial\Omega_r\cap\partial\Omega$ comprise the Stokes top and sides, Brinkman sides, and Darcy bottom and sides. We use $n=(0,-1)^T$ and $t=(1,0)^T$ at both interfaces, with $n$ directed downward. Interface endpoints must match the prescribed corner data. \Cref{fig:geometry} summarizes the geometry and boundary/interface conditions.

\begin{figure}[H]
\centering
\includegraphics[width=\linewidth]{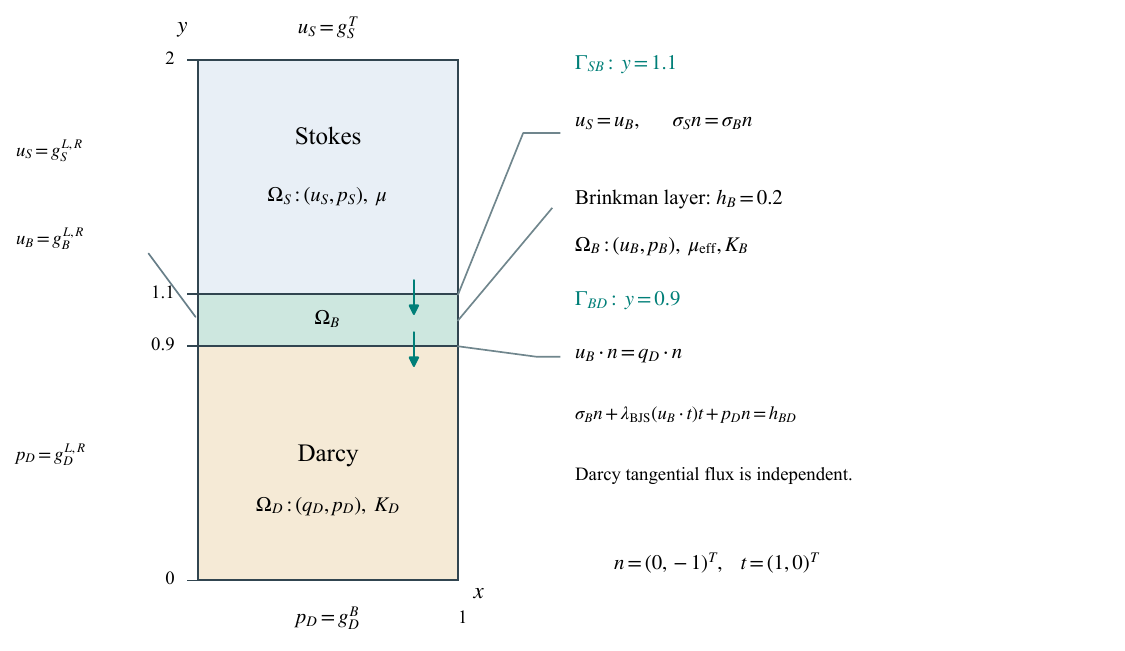}
\caption{Three-layer geometry for the manufactured problems, with prescribed exterior velocities and Darcy pressure and coupling-determined interface traces. Both interface normals point downward; only normal flux is continuous at Brinkman--Darcy.}
\label{fig:geometry}
\end{figure}

\subsection{Subdomain equations}

The original velocity--pressure formulation consists of the Stokes equations, the Brinkman equations with porous drag, and Darcy law with mass conservation:

\begin{subequations}\label{eq:original-sbd}
\begin{align}
-\mu\Delta u_{\mathrm S}+\nabla p_{\mathrm S}&=f_{\mathrm S},
&\nabla\cdot u_{\mathrm S}&=0 &&\text{in }\OS,\label{eq:original-stokes}\\
-\mu_{\mathrm{eff}}\Delta u_{\mathrm B}+\frac{\mu}{K_B}u_{\mathrm B}+\nabla p_{\mathrm B}&=f_{\mathrm B},
&\nabla\cdot u_{\mathrm B}&=0 &&\text{in }\OB,\label{eq:original-brinkman}\\
q_D&=-\frac{K_D}{\mu}\nabla p_D,
&\nabla\cdot q_D&=f_D &&\text{in }\OD.\label{eq:original-darcy}
\end{align}
\end{subequations}

The upper-region unknowns are velocity \(u_r:\Omega_r\to\mathbb R^2\) and
pressure \(p_r:\Omega_r\to\mathbb R\), where
\(r\in\{\mathrm S,\mathrm B\}\).  We use the nonsymmetric pseudo-stress
\begin{equation}
\sigma_r=\mu_r\nabla u_r-p_rI,\qquad
\mu_{\mathrm S}=\mu,\quad \mu_{\mathrm B}=\mu_{\mathrm{eff}}.
\label{eq:pseudostress}
\end{equation}
Hereafter, native traction is computed from the independently predicted pseudo-stress, whereas reconstructed traction uses $\mu_r\nabla u_r-p_rI$ evaluated from the predicted velocity and pressure.

The divergence is row-wise: $(\nabla\cdot\sigma)_i=\partial_x\sigma_{ix}+\partial_y\sigma_{iy}$. For constant viscosity and incompressibility, symmetric Cauchy stress has the same divergence but different interface traction; the pseudo-stress convention is used throughout.

For constant viscosity, the row-wise divergence satisfies
\begin{equation}
\nabla\cdot\sigma_r=\mu_r\Delta u_r-\nabla p_r,\qquad r\in\{\mathrm S,\mathrm B\}.
\label{eq:pseudostress-divergence}
\end{equation}
Thus the second-order viscous term and pressure gradient in \cref{eq:original-stokes,eq:original-brinkman} combine into $-\nabla\cdot\sigma_r$. Treating \cref{eq:pseudostress} as an independent constitutive equation gives the following mixed systems with first derivatives only.

In the Stokes region,
\begin{subequations}
\label{eq:stokes}
\begin{align}
\sigma_{\mathrm S}-\mu\nabla u_{\mathrm S}+p_{\mathrm S}I&=0,
&&\text{in }\OS,\\*
-\nabla\cdot\sigma_{\mathrm S}&=f_{\mathrm S},
&&\text{in }\OS,\\*
\nabla\cdot u_{\mathrm S}&=0,
&&\text{in }\OS.
\end{align}
\end{subequations}
In the transition region, the Brinkman system is
\begin{subequations}
\label{eq:brinkman}
\begin{align}
\sigma_{\mathrm B}-\mu_{\mathrm{eff}}\nabla u_{\mathrm B}
 +p_{\mathrm B}I&=0,
&&\text{in }\OB,\\*
-\nabla\cdot\sigma_{\mathrm B}+\frac{\mu}{K_B}u_{\mathrm B}
 &=f_{\mathrm B},
&&\text{in }\OB,\\*
\nabla\cdot u_{\mathrm B}&=0,
&&\text{in }\OB.
\end{align}
\end{subequations}
The Darcy region uses an independent flux instead of pseudo-stress. Eliminating $q_D$ from \cref{eq:original-darcy} gives the second-order pressure equation $-\nabla\cdot[(K_D/\mu)\nabla p_D]=f_D$; retaining $q_D$ gives the following first-order flux--pressure system.
The lower-region unknowns are Darcy flux \(q_D:\OD\to\mathbb R^2\) and
pressure \(p_D:\OD\to\mathbb R\):
\begin{subequations}
\label{eq:darcy}
\begin{align}
q_D+\frac{K_D}{\mu}\nabla p_D&=0,
&&\text{in }\OD,\\*
\nabla\cdot q_D&=f_D,
&&\text{in }\OD.
\end{align}
\end{subequations}
The scalar permeabilities \(K_B,K_D>0\) and viscosities
\(\mu,\mu_{\mathrm{eff}}>0\) are constant in the present study.

\subsection{Exterior and interface data}

Manufactured Dirichlet data are prescribed on the exterior:
\begin{equation}
u_{\mathrm S}=g_{\mathrm S}\quad\text{on }\GextS,\qquad
u_{\mathrm B}=g_{\mathrm B}\quad\text{on }\GextB,\qquad
p_D=g_D\quad\text{on }\GextD.
\label{eq:exterior-data}
\end{equation}
No Darcy pressure is prescribed on \(\GBD\); it is an unknown trace determined
through the coupling.

The selected Stokes--Brinkman conditions are continuity of velocity and
pseudo-traction,
\begin{subequations}
\label{eq:sb-interface}
\begin{align}
u_{\mathrm S}-u_{\mathrm B}&=0,
&&\text{on }\GSB,\\
\sigma_{\mathrm S}n-\sigma_{\mathrm B}n&=0,
&&\text{on }\GSB.
\end{align}
\end{subequations}
This selects the zero traction-jump member of the SBD model family.

At the Brinkman--Darcy interface, normal mass continuity and the BJS/normal
force balance are
\begin{subequations}
\label{eq:bd-interface}
\begin{align}
u_{\mathrm B}\cdot n-q_D\cdot n&=0,
&&\text{on }\GBD,\\
\sigma_{\mathrm B}n+\lambda_{\mathrm{BJS}}
(u_{\mathrm B}\cdot t)t+p_Dn-h_{BD}&=0,
&&\text{on }\GBD,
\end{align}
\end{subequations}
where
\begin{equation}
\lambda_{\mathrm{BJS}}
=\frac{\alpha\mu_{\mathrm{eff}}}{\sqrt{K_D}}.
\label{eq:bjs-coefficient}
\end{equation}
In global coordinates, an equivalent zero-residual form of
\cref{eq:bd-interface} is
\begin{equation}
(u_{\mathrm B})_y-(q_D)_y=0,\qquad
\sigma_{\mathrm B}n+
\begin{bmatrix}
\lambda_{\mathrm{BJS}}(u_{\mathrm B})_x\\
-p_D
\end{bmatrix}
-h_{BD}=0.
\label{eq:bd-global}
\end{equation}
Here \(h_{BD}=(h_{BD,x},h_{BD,y})^T\) is stored in the fixed global
\((x,y)\) basis; its tangential and normal components are
\(t^Th_{BD}\) and \(n^Th_{BD}\), respectively.  For each manufactured
example, the body forces and \(h_{BD}\) are generated by substituting the
chosen exact fields into
\cref{eq:stokes,eq:brinkman,eq:darcy,eq:bd-interface}.  They are not
independently tuned targets.

\subsection{Dimensionless model parameters and pressure anchoring}

The model is specified directly in dimensionless form, including its coordinates, fields, coefficients and sources. The nominal values in \cref{tab:physical-parameters} are prescribed model parameters rather than conversions of a particular dimensional dataset. The factors $1+\mu/K_B$ and $1+\lambda_{\mathrm{BJS}}$ introduced below provide numerical scaling of these dimensionless residuals.

\begin{table}[!htbp]\centering\small
\caption{Nominal MMS parameters; geometry is given in Figure~\ref{fig:geometry}.}
\label{tab:physical-parameters}
\begin{adjustbox}{max width=\linewidth}
\begin{tabular}{@{}cccccc@{}}\toprule
$\mu=\mu_{\mathrm{eff}}$ & $K_B=K_D$ & $\alpha$ & $\lambda_{\mathrm{BJS}}$ & $\mu/K_D$ & $(y_0,y_{BD},y_{SB},y_T)$\\\midrule
$1$ & $10^{-2}$ & $0.1$ & $1$ & $100$ & $(0,0.9,1.1,2)$\\
\bottomrule\end{tabular}
\end{adjustbox}\end{table}

The pressure levels are fixed by the coupled boundary-value problem rather
than by post-processing.  The exterior Darcy Dirichlet data anchor \(p_D\);
the normal component of the Brinkman--Darcy traction balance then anchors
\(p_{\mathrm B}\); and the Stokes--Brinkman traction balance anchors
\(p_{\mathrm S}\).  Consequently, the reported pressure errors are computed
without subtracting a fitted subdomain mean.  

\section{Hard-trace first-order residual method}\label{sec:method}
\Cref{fig:architecture} follows the construction from network outputs to constrained fields, residual training and pressure correction. We first define the mixed state and the field maps used at each interface.

\begin{figure}[p]\centering
\includegraphics[page=1,width=\linewidth]{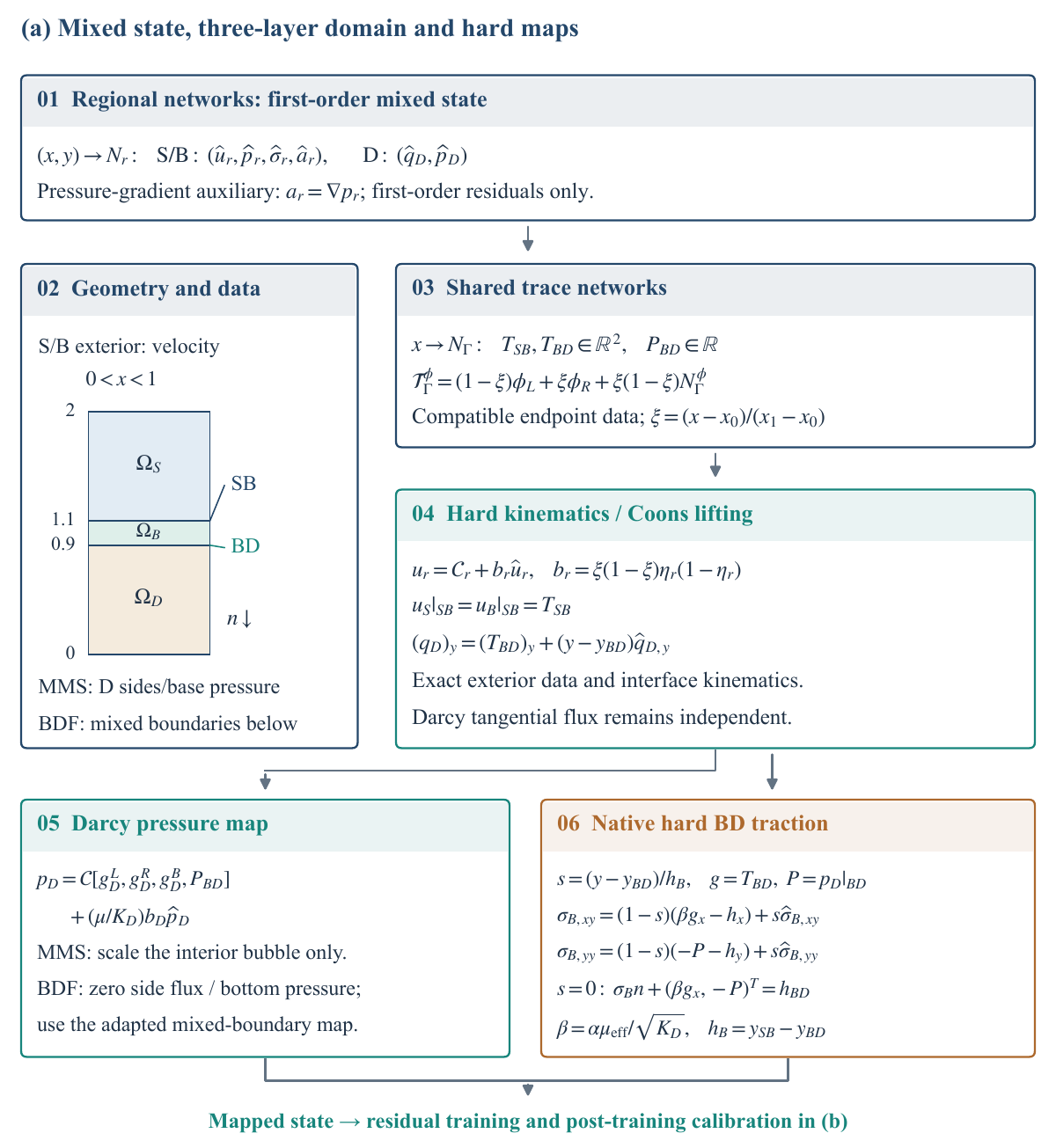}
\caption{Complete workflow (a): mixed state, shared traces and hard maps. Coons and Darcy pressure-scale formulas apply to MMS; BDF uses the mixed-boundary adaptation in Section~\ref{sec:bdf}.}\label{fig:architecture}
\end{figure}

\begin{figure}[p]\ContinuedFloat\centering
\includegraphics[page=2,width=\linewidth]{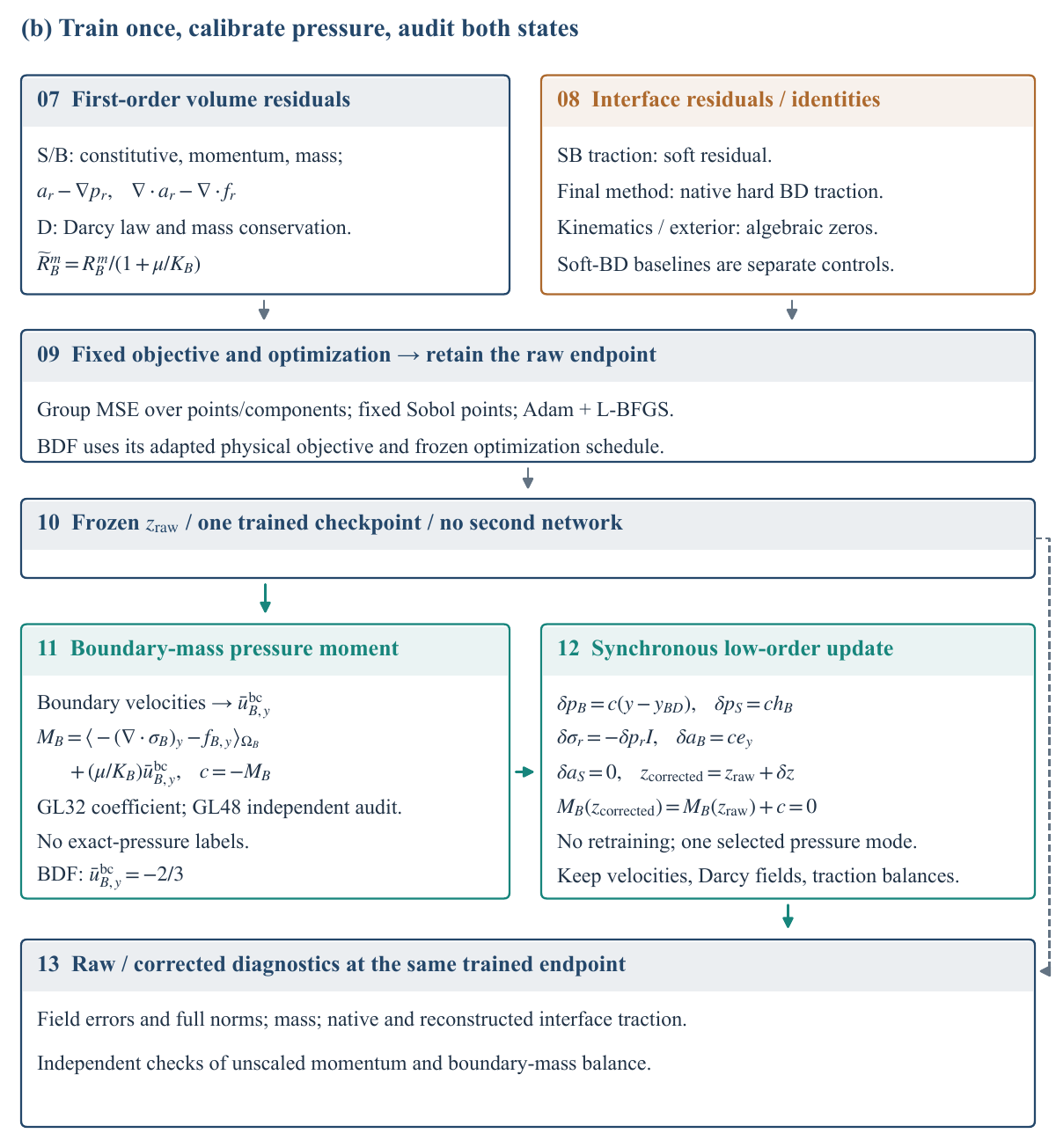}
\caption{Complete workflow (b): training, boundary-mass-consistent pressure correction and paired diagnostics. Raw/corrected share one trained model; the dashed path retains the raw state.}
\end{figure}
\subsection{First-order state and pressure auxiliaries}

\begin{equation}
z=\big(
u_{\mathrm S},p_{\mathrm S},\sigma_{\mathrm S},a_{\mathrm S},
u_{\mathrm B},p_{\mathrm B},\sigma_{\mathrm B},a_{\mathrm B},
q_D,p_D
\big),
\label{eq:state}
\end{equation}The upper networks output velocity, pressure, pseudo-stress and a pressure-gradient auxiliary; the Darcy network outputs flux and pressure. Velocities and Darcy fields enter the hard maps below. Upper pressure, stress and auxiliaries are raw outputs in the baseline; the hard-BD stress lifting is specified in \cref{subsec:hard-bd-map}. Expanded component lists are given in Supplement S1.
Eliminating pseudo-stress from the upper momentum equations gives

\begin{equation}
-\mu_r\Delta u_r+\nabla p_r
+\chi_r\frac{\mu}{K_B}u_r=f_r,
\qquad
\chi_{\mathrm S}=0,\quad\chi_{\mathrm B}=1.
\label{eq:upper-second-order}
\end{equation}
Because the viscosities and Brinkman resistance are constant and
\(\nabla\cdot u_r=0\), taking the divergence gives
\begin{equation}
\Delta p_r=\nabla\cdot f_r.
\label{eq:pressure-poisson}
\end{equation}
Introducing \(a_r=\nabla p_r\) converts this second-order scalar relation into
the first-order pair
\begin{subequations}
\label{eq:pressure-aux}
\begin{align}
a_r-\nabla p_r&=0,\\
\nabla\cdot a_r-\nabla\cdot f_r&=0.
\end{align}
\end{subequations}
The analytically generated \(\nabla\cdot f_r\) supplies a direct pressure-gradient signal after momentum scaling. This augmentation requires smooth fields, incompressibility, and constant isotropic coefficients; other cases require a new derivation. Our construction retains first-order differentiation; gPINNs instead penalize derivatives of PDE residuals \cite{YuEtAl2022Gradient}.

\subsection{Shared traces and hard field maps}

Let
\[
\xi=\frac{x-x_0}{x_1-x_0}\in[0,1].
\]
A scalar or vector trace on \(y=y_\Gamma\) interpolates the compatible exterior side data at its endpoints:

\begin{equation}
\mathcal T_\Gamma^\phi(x)=
(1-\xi)\phi_L(y_\Gamma)+\xi\phi_R(y_\Gamma)
+\xi(1-\xi)N_\Gamma^\phi(x).
\label{eq:compatible-trace}
\end{equation}
The neural correction vanishes at both endpoints for any network parameters. The three traces are

\begin{equation}
T_{SB}(x)\in\mathbb R^2,\qquad
T_{BD}(x)\in\mathbb R^2,\qquad
P_{BD}(x)\in\mathbb R.
\label{eq:three-traces}
\end{equation}
\(T_{SB}\) is the shared SB velocity; \(T_{BD}\) is the Brinkman bottom velocity, whose vertical component is shared with Darcy flux; \(P_{BD}\) is the Darcy top pressure. Their endpoints use the corresponding prescribed side data. Network sizes and initialization are specified in \cref{sec:experiments}.

Consider a generic rectangle
\([x_0,x_1]\times[y_b,y_t]\), set
\(\eta=(y-y_b)/(y_t-y_b)\), and let \(L(y),R(y),B(x),T(x)\) be compatible
data on its left, right, bottom, and top edges.  The Coons extension is
\begin{equation}\begin{aligned}
\mathcal C[L,R,B,T](x,y)
={}&(1-\xi)L(y)+\xi R(y)+(1-\eta)B(x)+\eta T(x)
\\
&-\big[
(1-\xi)(1-\eta)c_{00}
+\xi(1-\eta)c_{10}
\\
&\hspace{4em}
+(1-\xi)\eta c_{01}+\xi\eta c_{11}
\big],
\label{eq:coons}
\end{aligned}\end{equation}
where
\[
\begin{aligned}
c_{00}&=L(y_b)=B(x_0), & c_{10}&=R(y_b)=B(x_1),\\
c_{01}&=L(y_t)=T(x_0), & c_{11}&=R(y_t)=T(x_1).
\end{aligned}
\]
are the four common corner values.
The interior bubble
\begin{equation}
b(x,y)=\xi(1-\xi)\eta(1-\eta)
\label{eq:rectangular-bubble}
\end{equation}
vanishes on all four edges.

With \(\eta_{\mathrm S}=(y-y_{SB})/(y_T-y_{SB})\) and
\(\eta_{\mathrm B}=(y-y_{BD})/(y_{SB}-y_{BD})\), the physical velocities are
\begin{subequations}
\label{eq:upper-hard-velocities}
\begin{align}
u_{\mathrm S,\theta}
&=\mathcal C[
g_{\mathrm S}^{L},g_{\mathrm S}^{R},
T_{SB},g_{\mathrm S}^{T}]
+b_{\mathrm S}\widehat u_{\mathrm S},
\label{eq:hard-stokes-velocity}\\
u_{\mathrm B,\theta}
&=\mathcal C[
g_{\mathrm B}^{L},g_{\mathrm B}^{R},
T_{BD},T_{SB}]
+b_{\mathrm B}\widehat u_{\mathrm B}.
\label{eq:hard-brinkman-velocity}
\end{align}
\end{subequations}
The superscripts \(L,R,T\) identify exterior edges and \(b_r=\xi(1-\xi)\eta_r(1-\eta_r)\). The shared \(T_{SB}\) gives velocity continuity on the entire SB interface. 

The Darcy flux is mapped as
\begin{equation}
q_{D,\theta}(x,y)=
\begin{bmatrix}
\widehat q_{D,x}(x,y)\\
(T_{BD}(x))_y+(y-y_{BD})\widehat q_{D,y}(x,y)
\end{bmatrix}.
\label{eq:darcy-flux-lifting}
\end{equation}
Its vertical trace equals the Brinkman trace at BD; the tangential flux remains independent.

\subsection{Darcy pressure-scale separation}

Separate the unknown Darcy top pressure from its scaled interior correction:
\begin{equation}
P_{BD}(x)=
(1-\xi)g_D(x_0,y_{BD})+\xi g_D(x_1,y_{BD})
+\xi(1-\xi)N_{BD}^{p}(x).
\label{eq:darcy-pressure-trace}
\end{equation}
Here $b_D=\xi(1-\xi)\eta_D(1-\eta_D)$ and $\eta_D=(y-y_0)/(y_{BD}-y_0)$. The pressure lifting and physical field are
\begin{equation}
C_D^p=\mathcal C[
g_D^L,g_D^R,g_D^B,P_{BD}].
\label{eq:darcy-pressure-coons}
\end{equation}
\begin{equation}
p_{D,\theta}(x,y)
=C_D^p(x,y)+\frac{\mu}{K_D}\,b_D(x,y)\widehat p_D(x,y).
\label{eq:pressure-split}
\end{equation}
Only the interior bubble is multiplied by $\mu/K_D$, avoiding amplification of the random interface-trace initialization. This is an optimization parameterization, not an assumption that the physical trace is $O(1)$.

\subsection{Residual objective and implementation}\label{subsec:residual-objective}

For \(r\in\{\mathrm S,\mathrm B\}\), define
\begin{subequations}
\label{eq:upper-residuals}
\begin{align}
\R_r^\sigma
&=\sigma_r-\mu_r\nabla u_r+p_rI,
&
\R_r^c&=\nabla\cdot u_r,
\\
\R_r^m
&=-\nabla\cdot\sigma_r
+\chi_r\frac{\mu}{K_B}u_r-f_r,
&
\R_r^g&=a_r-\nabla p_r,
\\
\R_r^p
&=\nabla\cdot a_r-\nabla\cdot f_r.
\end{align}
\end{subequations}
The Brinkman momentum residual is normalized as
\begin{equation}
\widetilde{\R}_{\mathrm S}^{m}=\R_{\mathrm S}^{m},
\qquad
\widetilde{\R}_{\mathrm B}^{m}
=\frac{\R_{\mathrm B}^{m}}{1+\mu/K_B}.
\label{eq:momentum-scaling}
\end{equation}
The Darcy residuals are
\begin{equation}
\R_D^q=q_D+\frac{K_D}{\mu}\nabla p_D,\qquad
\R_D^c=\nabla\cdot q_D-f_D.
\label{eq:darcy-residuals}
\end{equation}

All four interface diagnostics are evaluated:
\begin{subequations}
\label{eq:interface-residuals}
\begin{align}
\begin{aligned}
\R_{SB}^{u}&=u_{\mathrm S}-u_{\mathrm B},\\
\R_{SB}^{t}&=\sigma_{\mathrm S}n-\sigma_{\mathrm B}n,
\end{aligned}\\
\begin{aligned}
\R_{BD}^{m}&=(u_{\mathrm B})_y-(q_D)_y
=-\big(u_{\mathrm B}\cdot n-q_D\cdot n\big),\\
\R_{BD}^{t}&=\sigma_{\mathrm B}n
+\lambda_{\mathrm{BJS}}(u_{\mathrm B}\cdot t)t+p_Dn-h_{BD}.
\end{aligned}
\end{align}
\end{subequations}
The kinematic groups vanish algebraically. For the soft-BD control, the traction residual uses tangential scaling

\begin{equation}
D_{BD}=\frac{1}{1+\lambda_{\mathrm{BJS}}}\,tt^T+nn^T,
\qquad
\widetilde{\R}_{BD}^{t}=D_{BD}\R_{BD}^{t}.
\label{eq:bd-traction-scaling}
\end{equation}
This leaves the BD normal-traction residual at unit weight; the bulk pressure mode is controlled separately in \cref{subsec:pressure-control}.

Exterior velocity residuals are $u_r-g_r$ and the Darcy pressure residual is $(p_D-g_D)/(\mu/K_D)$; all vanish for the hard map.
Their normalization is inactive for the hard exterior map.

Each residual group is averaged over both points and components. For \(R:G\to\mathbb R^{m_R}\), let \(\nu_G\) be normalized area or interface-arclength measure, or the prescribed mixture of edgewise measures for an exterior group. Define

\begin{equation}
\M_G(R)
=\frac{1}{m_R}\int_G |R(x)|_2^2\,d\nu_G(x).
\label{eq:continuous-group-mse}
\end{equation}
For a connected volume or interface, this is
\((m_R|G|)^{-1}\int_G|R|_2^2\).
For points \(X_G=\{x_i\}_{i=1}^{N_G}\), its implemented counterpart is
\begin{equation}
\MN_G(R)
=\frac{1}{m_RN_G}
\sum_{i=1}^{N_G}\sum_{j=1}^{m_R}R_j(x_i)^2.
\label{eq:empirical-group-mse}
\end{equation}
The factor \(1/m_R\) gives equal group coefficients after physical scaling. Exterior groups use the specified edge allocations, which need not be proportional to edge length; their residuals vanish for the hard model.

The continuous target represented by the complete code is
\begin{equation}\begin{aligned}
\mathcal J(z)={}&
\sum_{r\in\{\mathrm S,\mathrm B\}}
\big[
\M_{\Omega_r}(\R_r^\sigma)
+\M_{\Omega_r}(\widetilde{\R}_r^m)
+\M_{\Omega_r}(\R_r^c)
\\
&\hspace{7.0em}
+\M_{\Omega_r}(\R_r^g)
+\M_{\Omega_r}(\R_r^p)
\big]
\\
&+\M_{\OD}(\R_D^q)
+\M_{\OD}(\R_D^c)
+\M_{\GSB}(\R_{SB}^{t})
+\M_{\GBD}(\widetilde{\R}_{BD}^{t})
\\
&+\M_{\GextS}(\R_{\mathrm S}^{\mathrm{ext}})
+\M_{\GextB}(\R_{\mathrm B}^{\mathrm{ext}})
+\M_{\GextD}(\R_D^{\mathrm{ext}}).
\label{eq:continuous-loss}
\end{aligned}\end{equation}
The training objective \(J_N(\theta)\) replaces \(\M_G\) by \(\MN_G\) on fixed Sobol points. The scalings in \cref{eq:momentum-scaling,eq:bd-traction-scaling} are fixed weights. The identically zero kinematic groups are recorded separately; exterior groups remain in the objective but vanish. Thus loss magnitudes are not directly comparable across formulations.

\subsection{Hard BD traction map}\label{subsec:hard-bd-map}
Let $s=(y-y_{BD})/h_B$, $g=T_{BD}$, $P=P_{BD}$ and $\beta=\lambda_{\mathrm{BJS}}=\alpha\mu_{\mathrm{eff}}/\sqrt{K_D}$. For the Cartesian components $(h_x,h_y)$ of the prescribed BD load, set
\begin{equation}
\begin{aligned}
 \sigma_{B,xy}&=(1-s)(\beta g_x-h_x)+s\widehat\sigma_{B,xy},\\
 \sigma_{B,yy}&=(1-s)(-P-h_y)+s\widehat\sigma_{B,yy}.
\end{aligned}
\label{eq:extreme-hard-traction}
\end{equation}
At $s=0$, $n=(0,-1)$ gives $\sigma_Bn+(\beta g_x,-P)^T=h_{BD}$ identically; the original stress output remains free at $s=1$. Other stress components are unchanged. The traces $g$ and $P$ are learned; for the reported parameter settings, the prescribed MMS1 and MMS2 loads are zero and $(-0.8,0)^T$, respectively. Native BD traction is therefore exact, while traction reconstructed from $(u,p)$ still requires an independent check. This map is used in the final complete-method validation configuration; baseline kinematic hard traces retain soft BD traction, while the full BDF method uses the mixed-boundary-adapted hard BD map in Section~\ref{sec:bdf}.

\subsection{Consistent low-order pressure correction}\label{subsec:pressure-control}
Momentum scaling preserves consistency but weakly controls some pressure directions. With $h_B=y_{SB}-y_{BD}$, keep velocities and Darcy fields fixed and set $\delta p_B=c(y-y_{BD})$, $\delta p_S=ch_B$, $\delta\sigma_r=-\delta p_rI$, and $\delta a_B=ce_y$. This direction preserves interface, constitutive and auxiliary-consistency residuals, adding only $ce_y$ to Brinkman momentum. Starting from an exact solution, its scaled loss is $c^2/[2(1+\mu/K_B)^2]$. \Cref{fig:pressure-mechanism}(a,b) shows the pressure shape and this analytical sensitivity: as $K_B$ decreases, a fixed pressure perturbation contributes progressively less to the scaled objective. A small scaled loss alone therefore does not exclude this pressure error.

To control this direction, integrated incompressibility and prescribed top/side velocities determine the Brinkman mean $\bar u_{B,y}^{\mathrm{bc}}$ (formula and proof in Supplement S6). With brackets denoting volume averages, define the physical moment
\begin{equation}
 M_B(z)=\left\langle-(\nabla\!\cdot\sigma_B)_y-f_{B,y}\right\rangle_{\Omega_B}
       +\frac{\mu}{K_B}\bar u_{B,y}^{\mathrm{bc}},\qquad c=-M_B(z).
\label{eq:pressure-control-moment}
\end{equation}
After training, pressure, stress and auxiliaries are updated consistently along the direction above. Since $M_B(z+\delta z)=M_B(z)+c$, the correction controls this pressure direction without exact-pressure labels. Velocities, flux, Darcy pressure, BD tractions and SB traction jumps remain unchanged for both native and reconstructed stress; the two SB tractions shift equally. Fixed order-32 Gauss--Legendre integration determines the coefficient; order 48 checks it independently.

The corrected approximation is a neural field plus a known low-order basis, distinguished from the raw network output. Boundary mass balance supplies $\bar u_{B,y}^{\mathrm{bc}}$, not interior velocity labels; for an approximate field, $M_B$ differs from its raw mean momentum residual. This correction controls a selected integral pressure direction; its relation to local momentum residuals is detailed in Supplement S6. \Cref{subsec:spatial} evaluates pressure recovery using \cref{fig:pressure-mechanism}, with repeated-run permeability results in \cref{subsec:coverage}.

\subsection{Construction properties and analytical scope}

\begin{proposition}[Algebraic admissibility]
\label{prop:admissibility}
Assume that prescribed edge data agree at every corner of the three-rectangle
partition.  For every finite parameter vector, the fields in
\cref{eq:hard-stokes-velocity,eq:hard-brinkman-velocity,%
eq:darcy-flux-lifting,eq:pressure-split}
satisfy \cref{eq:exterior-data},
\(u_{\mathrm S}=u_{\mathrm B}\) on \(\GSB\), and
\((u_{\mathrm B})_y=(q_D)_y\) on \(\GBD\).
\end{proposition}

\begin{proof}
The correction in \cref{eq:compatible-trace} vanishes at the endpoints. On each edge, the bubble vanishes and the Coons patch equals its assigned trace. The shared \(T_{SB}\) enforces SB continuity; \cref{eq:darcy-flux-lifting} gives \((q_D)_y=(T_{BD})_y=(u_{\mathrm B})_y\) at BD.
\end{proof}

\begin{proposition}[Continuous consistency of the augmentation]
\label{prop:consistency}
Let a sufficiently smooth solution of
\cref{eq:stokes,eq:brinkman,eq:darcy,eq:exterior-data,%
eq:sb-interface,eq:bd-interface}
satisfy the stated constant-coefficient assumptions, and set
\(a_r=\nabla p_r\).  Then every term in
\(\mathcal J(z)\) vanishes and \(\mathcal J(z)=0\).
\end{proposition}

\begin{proof}
Substitution cancels the governing and interface residuals; \(a_r=\nabla p_r\) and \cref{eq:pressure-poisson} cancel the pressure auxiliaries. The hard identities follow from \cref{prop:admissibility}.
\end{proof}

These propositions establish algebraic admissibility and continuous consistency. A coercivity bound for the scaled residual operator yields conditional continuous-error control; the spaces, trace regularity and derivation are retained in Supplement S1. This conditional interpretation does not establish permeability-uniform stability, finite-sample error control or optimizer convergence. The construction uses rectangular interfaces and compatible corners.

\section{Experimental settings and evaluation protocol}\label{sec:experiments}
\subsection{Comparison configurations}
\Cref{tab:model-comparison} defines the baseline and component configurations. The complete method uses the hard-BD configuration followed by pressure correction. Throughout the experiments, \emph{raw} and \emph{corrected} denote the states before and after correction of one trained checkpoint.
\begin{table}[!htbp]\centering\footnotesize\setlength{\tabcolsep}{3pt}
\caption{Configuration definitions. Kinematic constraints include exterior data and interface sharing; the released control retains hard exterior data. BDF uses the mixed-boundary map and configuration in Section~\ref{sec:bdf}.}\label{tab:model-comparison}
\begin{adjustbox}{max width=\linewidth}
\begin{tabular}{@{}lccccr@{}}\toprule
Configuration & First order & Kinematics & Hard BD & Correction & Parameters\\\midrule
PINN baseline & No & Soft & No & None & 38,471\\
Soft first order & Yes & Soft & No & None & 39,121\\
Kinematic hard trace & Yes & Hard & No & Raw/corrected & 46,074\\
Hard BD & Yes & Hard & Yes & Raw/corrected & 46,074\\
Released interface & Yes & Exterior only & Yes & Raw/corrected & 46,074\\
\bottomrule
\end{tabular}
\end{adjustbox}\end{table}

We compare PINN, soft first order and kinematic hard traces using five seeds per configuration and manufactured problem. The complete method is tested in six MMS--permeability combinations: MMS1 and MMS2 at $K_B=10^{-2},10^{-4},10^{-6}$, with fixed $K_D=10^{-2}$ and three seeds per combination. Those seeds are reused across these combinations, so the 18 runs are not 18 independent replicates. Paired controls change either BD traction enforcement or interface sharing. BDF uses four validation runs, with its soft-BD control in Supplement S4. Supplement S8 reports correction results for 28 MMS benchmark and component checkpoints; seed identifiers and pairings are listed in Supplement S2.

\subsection{Networks, sampling and budgets}
All configurations use regional neural networks and fixed scrambled Sobol collocation. Hard-trace configurations additionally use interface networks. Architecture, initialization, point allocation and seed identifiers are specified in Supplement S2. Shared sampling controls comparison conditions but is not claimed optimal\cite{WuEtAl2023Sampling}.

Benchmark comparisons match nominal optimizer budgets within each case, with the hard-trace MMS2 run resetting L-BFGS history between two blocks. Complete-method and paired controls use their prespecified training schedules. BDF uses a separate three-block schedule with one fixed physical objective throughout. Exact optimizer settings and the earlier soft-BD staged protocol are given in Supplements S2, S4, S6 and S7.

Computational environments are listed by experimental batch in the environment table of Supplement S2; timing comparisons are restricted to the same hardware batch.

Baseline comparisons match nominal total budgets, not parameter count, optimization path or elapsed time. PINN reconstructs stress/flux and requires second derivatives; soft first order predicts stress/flux independently. Both baselines use unscaled unit-weight residuals and whole-field Darcy pressure scaling $\mu/K_D$. Kinematic hard traces add pressure auxiliaries, Darcy pressure separation and selected Brinkman-momentum/BD-tangential-traction scaling. Reporting these differences makes the scope of the baseline comparison explicit \cite{McGreivyHakim2024Baselines}. Whole-configuration gains cannot be assigned to any one component. 

\subsection{Manufactured solutions}
\label{subsec:manufactured-solutions}

Manufactured solutions make every source, boundary value, interface datum,
and field error available in closed form~\cite{Roache2002}.  In both examples,
the exact pseudo-stresses are defined by \cref{eq:pseudostress}, and all body
forces, divergence data, exterior values, and interface right-hand sides are
generated by substituting the fields below into the continuous equations.
They are not fitted independently.

\subsubsection{MMS1: exponential single-frequency field}

\noindent MMS1 adopts the exponential analytical fields of Ruan and Rybak~\cite{RuanRybak2026}, with source and interface data evaluated under the conventions of \cref{sec:model}.

Set \(\omega=1\), \(\phi=0\), and
\(\vartheta(x)=\omega x+\phi\).  The Stokes and Brinkman fields are the
restrictions of the same smooth upper solution,
\begin{subequations}
\label{eq:mms1-upper}
\begin{align}
u_{\mathrm S}^{\star}(x,y)=u_{\mathrm B}^{\star}(x,y)
&=
\begin{bmatrix}
\cos\vartheta(x)\\
\sin\vartheta(x)
\end{bmatrix}
\exp\!\bigl(\omega(y-y_{BD})\bigr),
\label{eq:mms1-velocity}\\
p_{\mathrm S}^{\star}(x,y)=p_{\mathrm B}^{\star}(x,y)
&=\sin\!\bigl(\omega(x+y-y_{BD})+\phi\bigr).
\label{eq:mms1-upper-pressure}
\end{align}
\end{subequations}
The Darcy pressure and flux are
\begin{subequations}
\label{eq:mms1-darcy}
\begin{align}
p_D^{\star}(x,y)
&=-\frac{\mu}{K_D}(y-y_{BD})\sin\vartheta(x),
\label{eq:mms1-darcy-pressure}\\
q_D^{\star}(x,y)
&=
\begin{bmatrix}
\omega(y-y_{BD})\cos\vartheta(x)\\
\sin\vartheta(x)
\end{bmatrix}
=-\frac{K_D}{\mu}\nabla p_D^{\star}(x,y).
\label{eq:mms1-darcy-flux}
\end{align}
\end{subequations}
The reference level in these formulas is specifically the
Darcy--Brinkman interface \(y_{BD}=0.9\), rather than an unspecified generic
vertical origin.

\subsubsection{MMS2: polynomial field with two-component interface flow}

MMS2 is a polynomial manufactured solution constructed in this work to test a different field structure and a nonzero linear Darcy interface-pressure trace.

Let \(\eta=y-y_{BD}\).  Again the exact Stokes and Brinkman fields are common:
\begin{subequations}
\label{eq:mms2-upper}
\begin{align}
u_{\mathrm S}^{\star}(x,y)=u_{\mathrm B}^{\star}(x,y)
&=
\begin{bmatrix}
0.20+\eta+\eta^2\\
x^2-x
\end{bmatrix},
\label{eq:mms2-velocity}\\
p_{\mathrm S}^{\star}(x,y)=p_{\mathrm B}^{\star}(x,y)
&=2\mu x+0.30\eta+0.20x\eta.
\label{eq:mms2-upper-pressure}
\end{align}
\end{subequations}
The exact Darcy fields are
\begin{subequations}
\label{eq:mms2-darcy}
\begin{align}
p_D^{\star}(x,y)
&=2\mu x+\frac{\mu}{K_D}
\left[x(1-x)\eta+\frac{\eta^3}{3}\right],
\label{eq:mms2-darcy-pressure}\\
q_D^{\star}(x,y)
&=
\begin{bmatrix}
-2K_D-(1-2x)\eta\\
-\bigl[x(1-x)+\eta^2\bigr]
\end{bmatrix}
=-\frac{K_D}{\mu}\nabla p_D^{\star}(x,y).
\label{eq:mms2-darcy-flux}
\end{align}
\end{subequations}
At \(\GBD\), both tangential Brinkman velocity and normal through-flow are
nonzero.  With \cref{tab:physical-parameters}, direct substitution gives
\(h_{BD}=(-0.8,0)^{\mathsf T}\).  Independent analytic tests of the governing and
interface equations give maximum absolute residuals of approximately
\(2.22\times10^{-16}\) for both manufactured examples.

\subsection{Independent evaluation and acceptance}

\label{subsec:independent-evaluation}\label{subsec:acceptance-rule}
All 30 baseline checkpoints are evaluated without retraining on an independent $81\times61$ tensor grid per region and 401 points per interface, using composite trapezoidal quadrature. Let $Q_r$ denote regional quadrature, $e_v=v_\theta-v^\star$, and $X_r$ the evaluation grid. Vector and gradient magnitudes are Euclidean and Frobenius, respectively. The reported errors are
\begin{equation}
E_{L^j,r}(v)
=\frac{Q_r[|e_v|^j]^{1/j}}
{Q_r[|v^{\star}|^j]^{1/j}},
\qquad j\in\{1,2\},
\label{eq:relative-lp-errors}
\end{equation}

\begin{equation}
E_{H^1,r}(v)
=\frac{\bigl(Q_r[|e_v|^2]+Q_r[|\nabla e_v|_F^2]\bigr)^{1/2}}
{\bigl(Q_r[|v^{\star}|^2]+Q_r[|\nabla v^{\star}|_F^2]\bigr)^{1/2}}.
\label{eq:relative-h1-error}
\end{equation}
Sampled maximum and Darcy $H(\mathrm{div})$ errors, including their normalizations, are defined in Supplement S2. 

The sampled $L_h^\infty$ is not a certified continuum maximum. No pressure mean is removed. Define $E_u=\max_{r=S,B}E_{L^2,r}(u_r)$, $E_p=\max_{r=S,B,D}E_{L^2,r}(p_r)$, and $E_q=E_{L^2,D}(q_D)$. Maxima are taken within each checkpoint before averaging seeds. Interface RMS is $(Q_\Gamma[|R|_2^2]/|\Gamma|)^{1/2}$ in physical, unscaled units.
\begin{table}[!htbp]\centering\small
\caption{Prespecified baseline acceptance criteria; interface RMS uses unscaled nondimensional physical quantities.}
\label{tab:acceptance-rule}
\begin{adjustbox}{max width=\linewidth}
\begin{tabular}{@{}lll@{}}\toprule
Quantity & Strict upper bounds & Applies to\\\midrule
Fields $E_u;\ E_p;\ E_q$ & $0.05;\ 0.10;\ 0.10$ & All methods\\
SB/BD traction vector RMS (each) & $0.10$ & All methods\\
SB velocity/BD mass RMS (each) & $0.05$ & Both soft baselines\\
Each hard exterior/kinematic identity RMS & $10^{-12}$ & Hard trace\\
\bottomrule\end{tabular}
\end{adjustbox}\end{table}

A checkpoint passes only when every applicable strict acceptance criterion holds. Baseline traction criteria use reconstructed stress for PINN and independent stress for first-order models; reconstructed first-order tractions are additionally checked, without changing the original criterion. Mean and sample standard deviation are computed over all five checkpoints, with denominator $n-1$; the pass count is checkpoint-based. The criteria are study-specific, not universal accuracy guarantees. Full quadrature and aggregation definitions are in Supplement S2. The extended MMS and BDF acceptance families, including their componentwise versus vector traction conventions, are listed in Supplement S2.

\section{Results and validation}\label{sec:results}
\FloatBarrier
\subsection{Experimental system and overall accuracy}\label{subsec:baseline-results}
We first compare flow fields and regional errors at $K_B=K_D=10^{-2}$. Subsequent tests vary Brinkman permeability, isolate individual components and examine boundary-driven filtration. Configuration definitions and sample sizes are given in \cref{sec:experiments}.

In \cref{fig:flow-fields}, the numerical solutions recover the principal flow directions and magnitude variations of both MMS1 and MMS2 across the three regions. Here $w=u$ in Stokes/Brinkman and $w=q_D$ in Darcy. The shared magnitude scale allows direct comparison with the exact fields, while the error maps locate the remaining discrepancies. Both cases use the fixed validation seed specified in Supplement S2.

\begin{figure}[!htbp]\centering
\includegraphics[width=\linewidth]{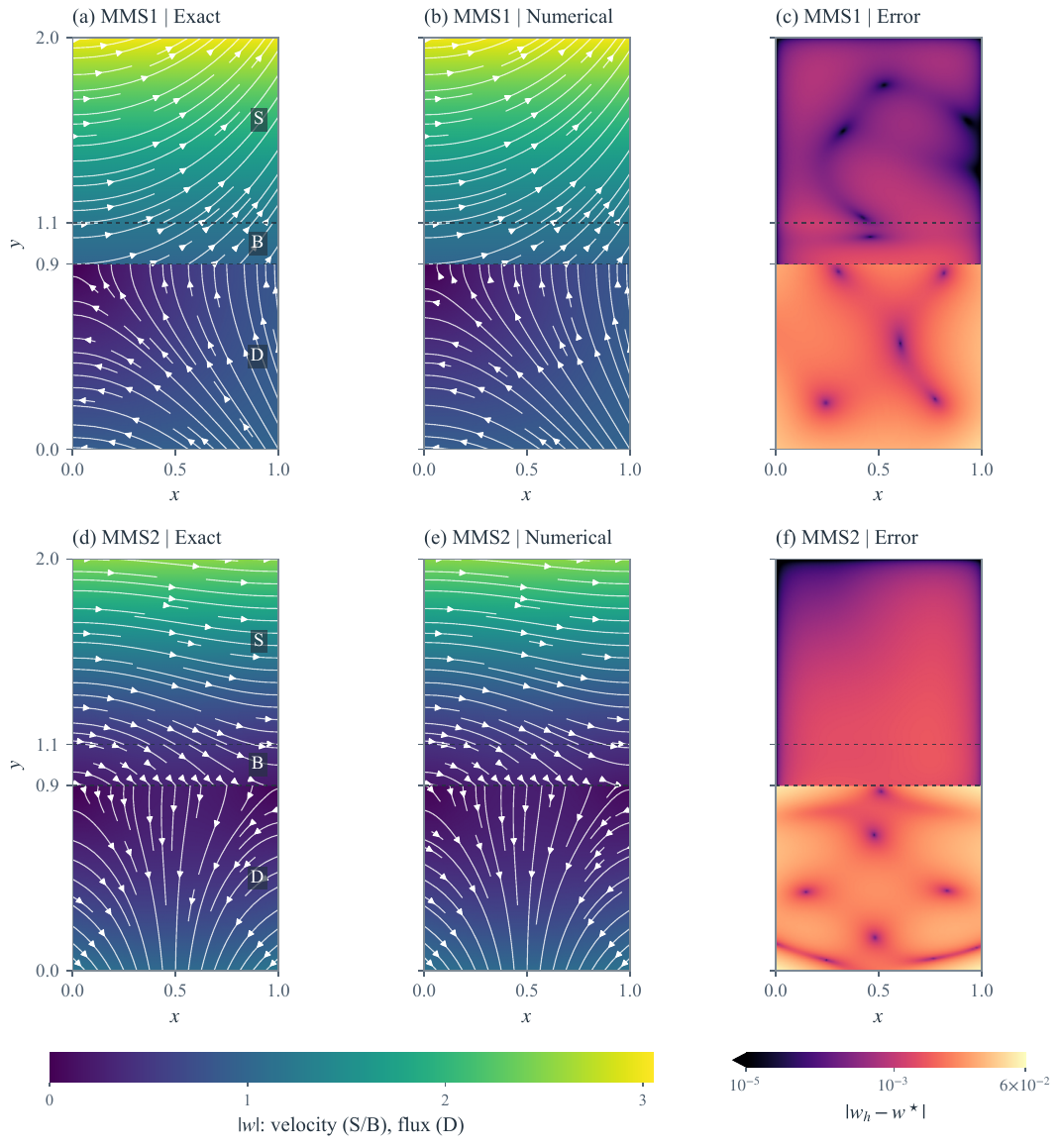}
\caption{Exact and numerical flow fields for MMS1 (a--c) and MMS2 (d--f), $K_B=K_D=10^{-2}$. Each row shows exact and numerical magnitudes with streamlines, followed by vector error $\lVert w_h-w^\star\rVert_2$, where $w=u$ in S/B and $w=q_D$ in D. Magnitudes share a linear scale; errors share a logarithmic scale floored at $10^{-5}$. Dashed lines mark SB and BD. Streamlines are computed regionwise, without joining tangential velocities across BD.}\label{fig:flow-fields}
\end{figure}

Across all six fields in \cref{fig:overall-accuracy}, kinematic hard traces give lower benchmark mean errors than both soft baselines for both test cases. The lower panels report the corrected complete method over three validation seeds. The two rows use distinct protocols (\cref{sec:experiments}); component effects are assessed by the paired controls below.

\begin{table}[!htbp]
\centering
\caption{Benchmark comparison group: mean $\pm$ sample SD over five seeds per method and case; errors in percent. Velocity and pressure maxima are taken across regions within each trained model before aggregation. Nominal budgets are matched, but parameter counts, scaling and optimization paths differ.}\label{tab:main-audit}
\footnotesize\setlength{\tabcolsep}{4pt}
\begin{adjustbox}{max width=\linewidth}
\begin{tabular}{@{}llrrrr@{}}
\toprule
Case & Method & Velocity & Pressure & Darcy flux & Passes\\
\midrule
\multirow{3}{*}{MMS1}
& PINN & \(8.584\pm1.106\) & \(32.016\pm13.532\) & \(2.186\pm0.392\) & 0/5\\
& Soft first-order & \(6.564\pm0.616\) & \(50.806\pm18.388\) & \(12.704\pm2.505\) & 0/5\\
& {Kinematic/raw} & \({0.223\pm0.081}\) & \({4.331\pm1.660}\)
& \({1.219\pm0.176}\) & {5/5}\\
\midrule
\multirow{3}{*}{MMS2}
& PINN & \(7.148\pm1.261\) & \(84.240\pm21.037\) & \(3.800\pm0.654\) & 0/5\\
& Soft first-order & \(6.068\pm0.857\) & \(67.379\pm9.135\) & \(22.597\pm1.303\) & 0/5\\
& {Kinematic/raw} & \({0.501\pm0.078}\) & \({7.381\pm1.099}\)
& \({1.603\pm0.109}\) & {5/5}\\
\bottomrule
\end{tabular}
\end{adjustbox}
\end{table}

Kinematic hard traces pass all joint acceptance criteria in 5/5 seeds for both cases; both soft baselines pass 0/5. For both MMS cases, the kinematic-hard-trace model has lower five-seed mean errors than both baselines for all six fields in the $L^1$, $L^2$, sampled $L_h^\infty$ and full $H^1$ norms, and for the Darcy flux in $H(\mathrm{div})$. Full norms and correction controls for benchmark checkpoints are retained in Supplements S2 and S8.

\begin{figure}[!htbp]\centering
\includegraphics[width=\linewidth]{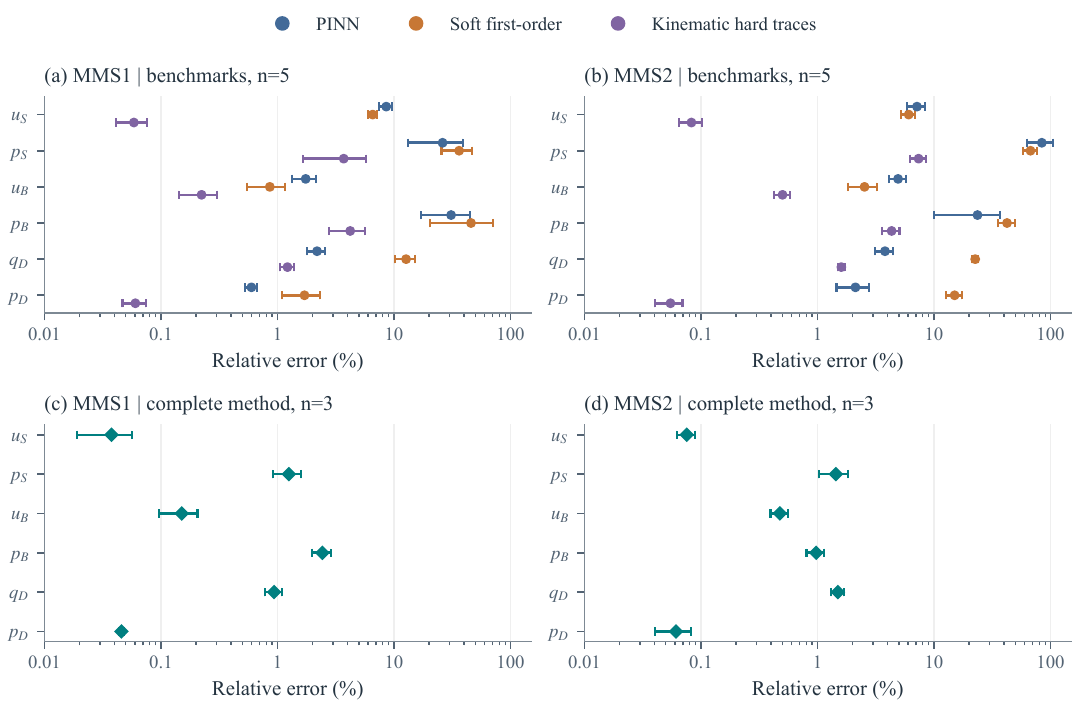}
\caption{Six-field accuracy at $K_B=K_D=10^{-2}$. Upper row: three benchmark methods, five seeds each; lower row: corrected complete method, three validation seeds; columns: MMS1 and MMS2. Points and intervals are means and sample SDs on a common logarithmic axis. Rows compare distinct configurations, with protocols in Section~\ref{sec:experiments}.}\label{fig:overall-accuracy}
\end{figure}

\subsection{Spatial pressure recovery and interface consistency}\label{subsec:spatial}
At $K_B=10^{-6}$, the scaled objective is weakly sensitive to the pressure direction in \cref{fig:pressure-mechanism}(a,b). For the MMS2 checkpoint in panels (c,d), the zero of the physical moment gives $c_*\approx0.707$. Applying this correction reduces Stokes- and Brinkman-pressure relative $L^2$ errors from 10.15\% and 6.39\% to 0.77\% and 1.29\%, respectively. Exact pressure is used only to evaluate these errors; the coefficient is determined from the moment. The diagnostic and run-selection protocol are given in Supplement S2.

In \cref{fig:pressure-contours}, corrected pressure contours lie closer to the exact contours over much of both upper regions. The adjustment is constant in Stokes and linear in $y$ in Brinkman, consistent with the direction derived in \cref{subsec:pressure-control}. Each raw/corrected pair uses the same contour levels.

\begin{figure}[p]\centering
\captionsetup{font=small,skip=5pt}
\includegraphics[width=0.98\linewidth]{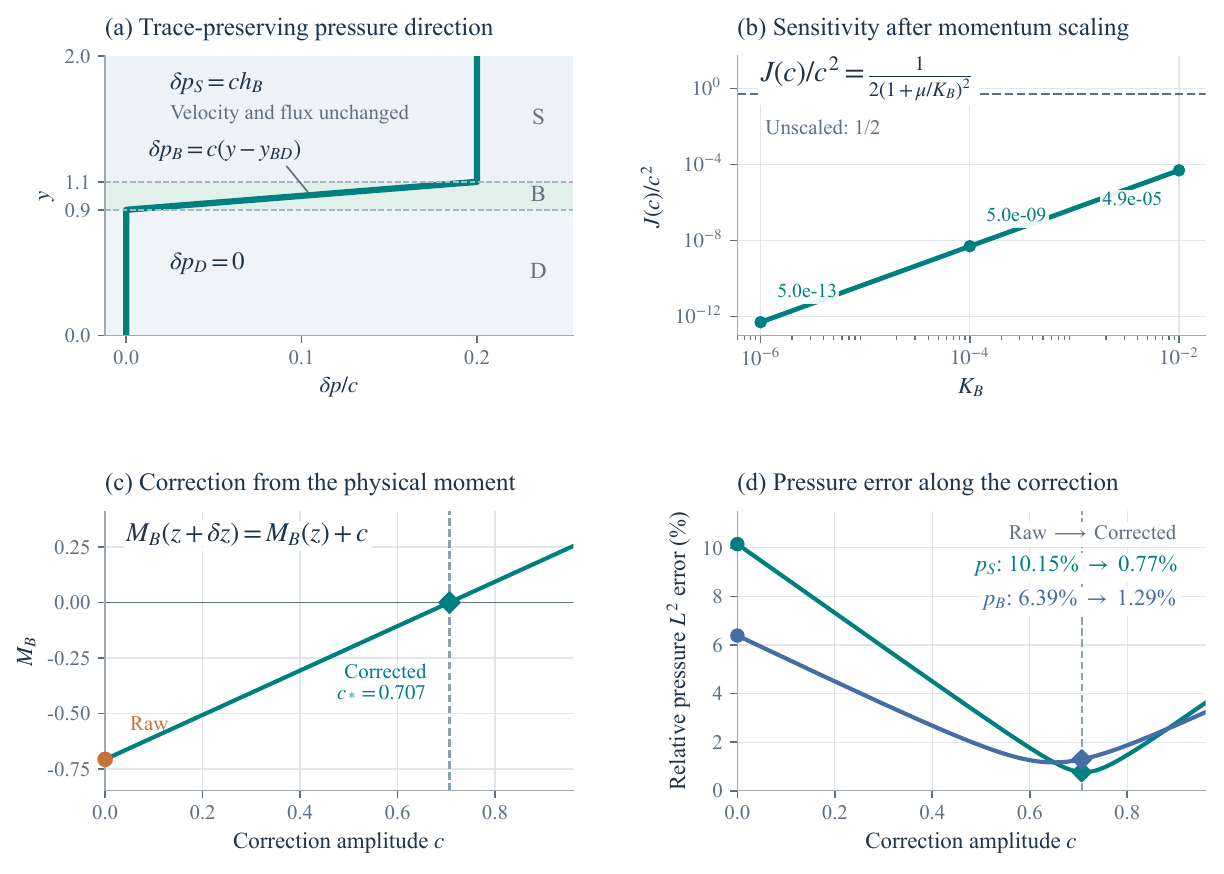}
\caption{Pressure correction. (a,b) Pressure direction and scaled loss coefficient, $\mu=1$; the dashed coefficient is $1/2$ without scaling. (c,d) Physical moment and pressure errors for MMS2 at $K_B=10^{-6}$, $K_D=10^{-2}$. Order 32 determines $c_*$; order 48 checks the moment. Diamonds mark $c_*$; exact pressure is used only for error evaluation.}\label{fig:pressure-mechanism}
\vspace{5pt}
\includegraphics[width=\linewidth]{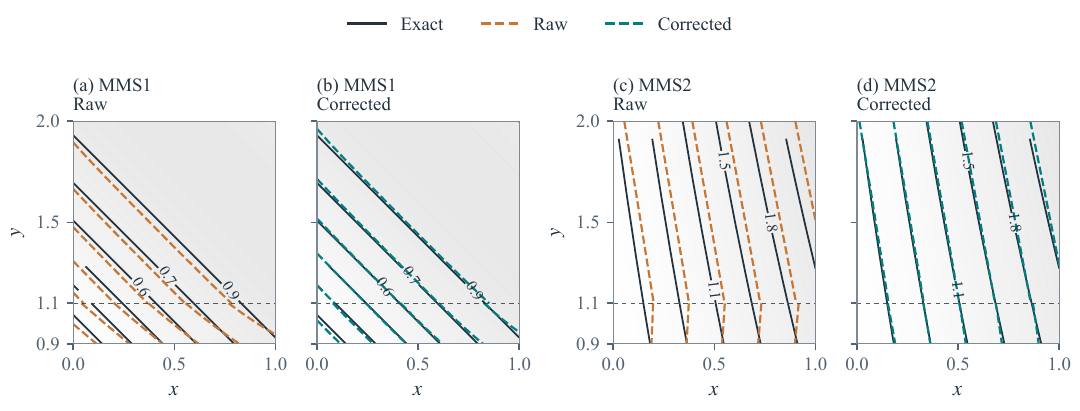}
\caption{S/B pressure contours at $K_B=K_D=10^{-2}$: MMS1 (a,b), MMS2 (c,d), with the seed used in \cref{fig:flow-fields}. Black: exact; orange: raw; teal: corrected. Each MMS pair shares contour levels; the horizontal dashed line marks SB.}\label{fig:pressure-contours}
\end{figure}

\Cref{fig:interface-profiles} examines the corrected pressure traces at SB and normal flow at BD for the same checkpoints. The two manufactured pressure traces coincide at SB for these tests; their numerical traces are plotted separately to expose the remaining pressure discrepancy. This agreement is a diagnostic of the pressure approximation, not an additional hard pressure-continuity condition. At BD, the shared normal-flow trace yields a maximum numerical jump of about $1.1\times10^{-16}$ in MMS1 and zero at stored precision in MMS2. The shared trace therefore enforces BD normal mass continuity to floating-point precision.

The cuts at $x=0.25,0.5,0.75$ in \cref{fig:coverage-profiles} show a smaller pressure offset after correction at all three locations in both manufactured cases.

\begin{figure}[!htbp]\centering
\captionsetup{font=small,skip=5pt}
\includegraphics[width=\linewidth]{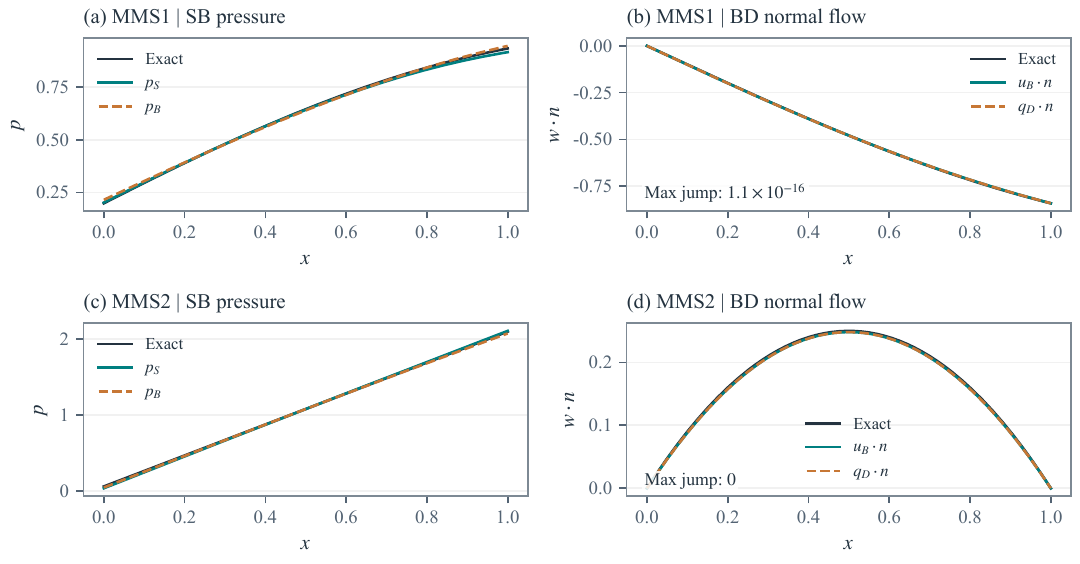}
\caption{Interface profiles for representative complete-method checkpoints: MMS1 (a,b), MMS2 (c,d), $K_B=K_D=10^{-2}$. Left: exact SB pressure and separate corrected S/B traces; right: exact and numerical BD normal flow, with $n=(0,-1)$ and separate $u_B\cdot n$ and $q_D\cdot n$ curves. Maximum jumps use the saved interface grid.}\label{fig:interface-profiles}
\end{figure}

\begin{figure}[!htbp]\centering
\captionsetup{font=small,skip=5pt}
\includegraphics[width=\linewidth]{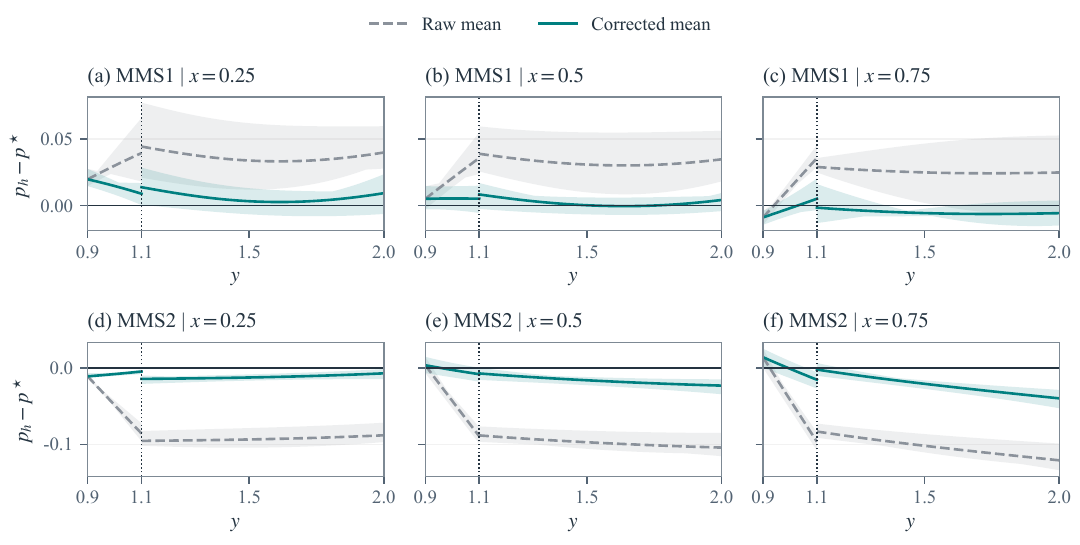}
\caption{Signed B/S pressure errors at $K_B=K_D=10^{-2}$: MMS1 (top), MMS2 (bottom), and cuts at $x=0.25,0.5,0.75$ (columns). Grey dashed/teal solid curves show raw/corrected three-seed means; bands span seed minima and maxima, not confidence intervals. Regions are drawn separately; the dotted line marks SB at $y=1.1$, and panels in each row share an error scale.}\label{fig:coverage-profiles}
\end{figure}

\FloatBarrier
\subsection{Permeability adaptation and spatial error}\label{subsec:coverage}
With one fixed training schedule across all six MMS--permeability combinations, the mean corrected upper-pressure $L^2$ errors range from 0.87\% to 2.58\%, and the mean full Brinkman-pressure $H^1$ errors from 4.77\% to 8.03\% (\cref{tab:prospective-coverage}). Order-48 quadrature verifies every order-32 correction coefficient, and all invariant checks pass. Exact fields stay fixed while forcing varies with permeability, so this experiment measures sensitivity to the coefficient and its associated optimization scaling.

\begin{table}[!htbp]\centering\small\setlength{\tabcolsep}{4pt}
\caption{Complete configuration in six MMS--permeability combinations: mean $\pm$ sample SD over three seeds per combination; errors in percent. Passes require all 14 acceptance criteria on both grids.}\label{tab:prospective-coverage}
\begin{adjustbox}{max width=\linewidth}
\begin{tabular}{@{}lccccc@{}}\toprule
Case, $K_B$ & State & $p_S$: $L^2$ & $p_B$: $L^2$ & $p_B$: $H^1$ & Pass\\\midrule
MMS1, $10^{-2}$ & raw & \(3.80\pm2.42\) & \(4.56\pm1.38\) & \(15.92\pm4.38\) & 3/3\\
MMS1, $10^{-2}$ & corrected & \(1.25\pm0.34\) & \(2.43\pm0.44\) & \(7.54\pm3.32\) & 3/3\\
MMS1, $10^{-4}$ & raw & \(4.23\pm2.65\) & \(5.32\pm2.13\) & \(18.52\pm7.01\) & 3/3\\
MMS1, $10^{-4}$ & corrected & \(1.49\pm0.09\) & \(2.58\pm0.68\) & \(8.03\pm2.45\) & 3/3\\
MMS1, $10^{-6}$ & raw & \(4.24\pm2.36\) & \(5.31\pm1.86\) & \(18.87\pm7.24\) & 3/3\\
MMS1, $10^{-6}$ & corrected & \(1.36\pm0.06\) & \(2.49\pm0.26\) & \(7.92\pm2.62\) & 3/3\\
MMS2, $10^{-2}$ & raw & \(7.03\pm0.97\) & \(4.50\pm0.59\) & \(20.17\pm2.56\) & 3/3\\
MMS2, $10^{-2}$ & corrected & \(1.44\pm0.41\) & \(0.97\pm0.17\) & \(5.29\pm0.53\) & 3/3\\
MMS2, $10^{-4}$ & raw & \(8.59\pm1.76\) & \(5.57\pm1.12\) & \(25.42\pm5.89\) & 2/3\\
MMS2, $10^{-4}$ & corrected & \(0.94\pm0.24\) & \(0.87\pm0.25\) & \(4.77\pm0.63\) & 3/3\\
MMS2, $10^{-6}$ & raw & \(8.46\pm1.73\) & \(5.51\pm1.11\) & \(24.98\pm5.77\) & 2/3\\
MMS2, $10^{-6}$ & corrected & \(0.93\pm0.24\) & \(0.97\pm0.28\) & \(4.81\pm0.71\) & 3/3\\
\bottomrule\end{tabular}
\end{adjustbox}\end{table}

The MMS2 maps in \cref{fig:prospective-coverage} show the spatial effect at all three permeabilities. Using one seed and a common signed-error scale, they expose the reduction of the broad pressure bias and the smaller spatial variation left after correction. Multi-seed results for both manufactured solutions are given in \cref{tab:prospective-coverage} and Supplement S7.

\begin{figure}[!htbp]\centering
\includegraphics[width=\linewidth]{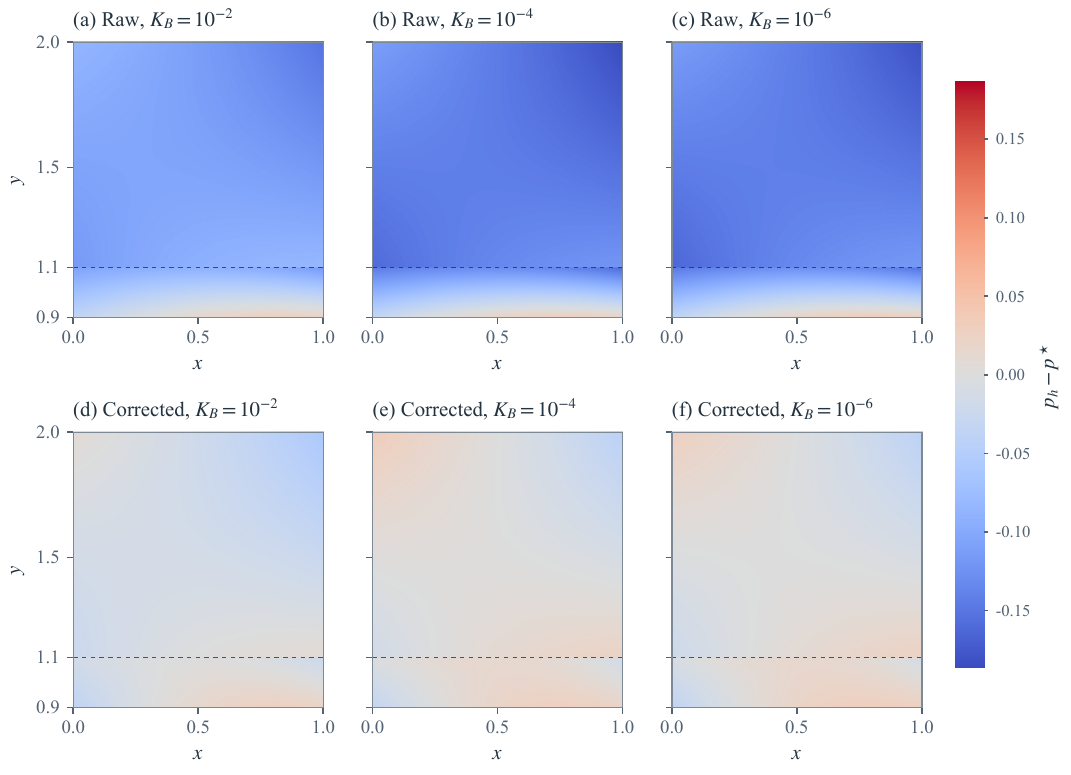}
\caption{MMS2 S/B pressure errors at $K_B=10^{-2},10^{-4},10^{-6}$, with $K_D=10^{-2}$. Top/bottom: raw/corrected states of each checkpoint, on one zero-centered scale for $p_h-p^\star$. One fixed representative seed is used across coefficients; maps are not seed averages. Exact fields stay fixed while forcing varies; dashed lines mark SB.}\label{fig:prospective-coverage}
\end{figure}

In every run across the six MMS--permeability combinations, correction reduces both upper-pressure $L^2$ errors, full Brinkman-pressure $H^1$ and the normal-derivative error; all corrected states pass the 14 acceptance criteria on both grids. Supplement S7 provides per-run results, derivative and unscaled-momentum diagnostics, and fine-grid checks. Passing the criteria does not imply equally small errors in every derivative component: for MMS2 at $K_B=10^{-6}$, the corrected mean full Brinkman-pressure $H^1$ error is 4.81\%, while the normal-derivative relative $L^2$ error remains 26.07\%. The correction controls the selected pressure moment and need not reduce local unscaled momentum residuals.

\subsection{Roles of key components and validation of pressure correction}\label{subsec:components}
\subsubsection{Hard BD traction}\label{subsec:extreme}
The component-control group uses three paired runs at each of the three MMS2 permeabilities; the same seed set is reused across coefficients (Supplement S2). Parameter counts, raw weights, collocation points and the 400+300+300 schedule are matched; the hard map changes the physical initial stress. \Cref{fig:bd-components}(a,b) isolates this contrast through uncorrected hard-BD minus soft-BD paired differences. At $K_B=10^{-4},10^{-6}$, both pressure $L^2$ errors and full Brinkman-pressure $H^1$ improve in every pair: hard BD passes 3/3 and soft BD 0/3, failing Stokes pressure. At $10^{-2}$, both pass 3/3; \cref{fig:bd-components} retains all paired differences across the three coefficients, with complete traction and cost records in Supplement S5.

Five additional independent hard-BD seeds at $K_B=10^{-6}$ pass all 14 acceptance criteria on both grids, providing further evidence of repeatability under strong drag. After correction, both BD configurations pass 3/3 in every paired-control cell. Corrected pressure rankings are assessed separately from their corresponding states in Supplement S8; complete uncorrected controls are retained in Supplement S5.

\subsubsection{Shared and released interfaces}\label{subsec:pressure-control-validation}
At MMS2 $K_B=10^{-6}$, three paired validation runs are used for shared and released interfaces, both with 46,074 parameters. Raw initialization, collocation points, auxiliaries, scaling, hard exterior boundaries, hard BD and optimizer budgets are matched. Release adds no parameters, although physical initial fields differ. Each trained model is evaluated before and after correction.

\begin{table}[!htbp]\centering\small\setlength{\tabcolsep}{4pt}
\caption{Three-seed mean $\pm$ sample standard deviation; errors are percentages. Full $H^1$ and the separate normal-derivative relative $L^2$ error are distinguished. Passes require the original 14 acceptance criteria on both grids.}
\label{tab:pressure-control}
\begin{adjustbox}{max width=\linewidth}
\begin{tabular}{@{}lccccc@{}}\toprule
State & $p_S$: $L^2$ & $p_B$: $L^2$ & $p_B$: $H^1$ & $\partial_y p_B$: $L^2$ & Passes\\\midrule
Hard/raw & \(8.06\pm1.35\) & \(5.19\pm0.70\) & \(25.03\pm4.22\) & \(146.35\pm24.81\) & 3/3\\
Hard/corrected & \(0.87\pm0.07\) & \(0.87\pm0.14\) & \(4.90\pm0.41\) & \(27.25\pm2.64\) & 3/3\\
Released/raw & \(7.93\pm1.16\) & \(5.03\pm0.50\) & \(26.17\pm4.55\) & \(153.13\pm26.80\) & 3/3\\
Released/corrected & \(1.15\pm0.16\) & \(1.21\pm0.50\) & \(4.90\pm0.53\) & \(27.11\pm3.35\) & 3/3\\
\bottomrule\end{tabular}
\end{adjustbox}\end{table}

\Cref{fig:bd-components}(c) and \cref{tab:pressure-control} show pressure gains in both interface configurations; corrected full $H^1$ means are both about 4.90\%. Hard sharing makes kinematic jumps zero. Released-interface SB velocity/BD normal-jump RMS means are 0.00083/0.00220, unchanged by correction. Sharing supplies exact kinematics, whereas pressure correction controls the associated weak direction; these mechanisms must be identified separately.
The nonlinear Darcy pressure-trace tests in Supplement S5 show a remaining limitation at the prescribed training budget. The split and amplified variants have mean trace $L^2$ errors of 22.63\% and 18.69\%, respectively, and neither passes the inherited diagnostic screen in any of its three runs. These results do not establish accurate nonlinear-trace learning or general superiority of pressure-scale separation; the screen is not an application-validated accuracy requirement.

\begin{figure}[!htbp]\centering
\includegraphics[width=\linewidth]{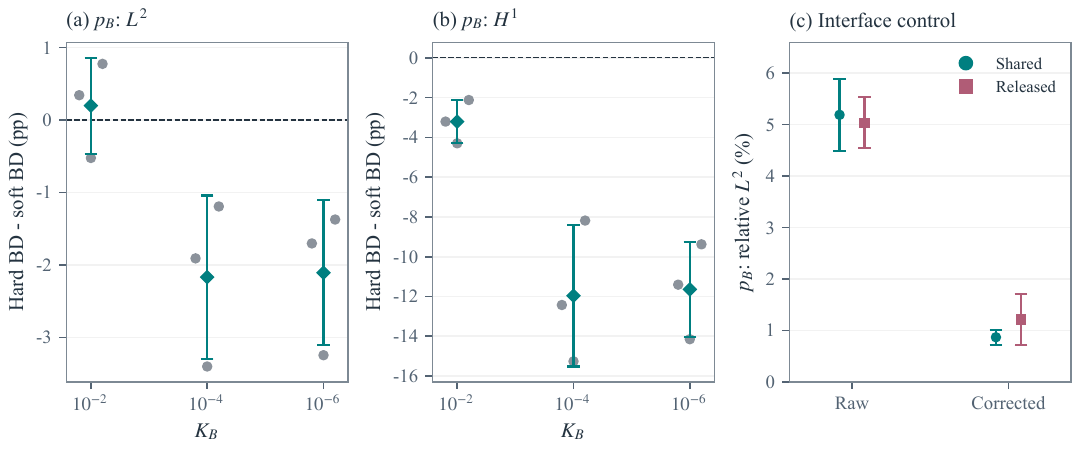}
\caption{Component controls. (a,b) Uncorrected hard-BD minus soft-BD paired differences in Brinkman-pressure relative $L^2$ and full $H^1$, in percentage points; negative values favor hard BD. Grey dots: three pairs per coefficient; teal diamonds/intervals: means/sample SDs. (c) Raw/corrected Brinkman-pressure $L^2$ means and sample SDs for shared/released interfaces in MMS2 at $K_B=10^{-6}$, three seeds per configuration.}\label{fig:bd-components}
\end{figure}

\FloatBarrier
\subsection{Boundary-driven filtration (BDF)}
\label{sec:bdf}
We now apply the method to inlet-driven filtration with mixed Darcy boundary conditions.

\paragraph{Problem and mixed-boundary trial map}

To test the construction without manufactured forcing, we consider a physical
filtration problem, denoted boundary-driven filtration (BDF).  The rectangles remain
\(\OD=(0,1)\times(0,0.9)\), \(\OB=(0,1)\times(0.9,1.1)\), and
\(\OS=(0,1)\times(1.1,2)\), but now
\begin{equation}
 \mu=\mu_{\mathrm{eff}}=1,\qquad K_B=K_D=0.1,\qquad
 \alpha=0.1,\qquad \beta=\lambda_{\mathrm{BJS}}=\alpha\mu_{\mathrm{eff}}/\sqrt{K_D}.
 \label{eq:bdf-parameters}
\end{equation}
All body forces, mass sources, and manufactured interface corrections are
zero.  The upper equations are
\(-\operatorname{div}\sigma_r+\chi_{r=B}10u_r=0\) and
\(\operatorname{div}u_r=0\), with \(\sigma_r=\nabla u_r-p_r I\).
Darcy flow satisfies \(q_D+0.1\nabla p_D=0\) and
\(\operatorname{div}q_D=0\).  At the top,
\(u_S=(0,-4x(1-x))\); the upper side walls have zero velocity.  Darcy
side walls have \(q_x=0\), while the bottom has \(p_D=0\).
With \(n=(0,-1)\) and \(t=(1,0)\), the SB velocity and traction are
continuous, and the BD conditions are
\begin{equation}
 u_B\cdot n=q_D\cdot n,\qquad
 \sigma_B n+\beta(u_B\cdot t)t+p_D n=0.
 \label{eq:bdf-interface}
\end{equation}
The inlet supplies a total downward flux of \(2/3\).  The geometry is that of \cref{fig:geometry}, with the exterior data replaced as stated above.

This is a non-manufactured numerical benchmark for the stated continuum
model. The permeability is fixed at 0.1 and the inlet amplitude at one;
the experiment evaluates this selected medium-resistance configuration.

The mixed Darcy boundary conditions require a different lifting from
\cref{sec:method}.  We use two velocity traces
\(T_{SB}=x(1-x)N_{SB}(x)\) and \(T_{BD}=x(1-x)N_{BD}(x)\).
Writing \(\eta_r=(y-y_r^-)/(y_r^+-y_r^-)\) for each upper rectangle,
\begin{equation}
 u_r=(1-\eta_r)T_r^-+\eta_r T_r^+
       +x(1-x)\eta_r(1-\eta_r)N_{u,r}(x,y),
 \label{eq:bdf-upper-map}
\end{equation}
where \((T_S^-,T_S^+)=(T_{SB},u_{\mathrm{in}})\) and
\((T_B^-,T_B^+)=(T_{BD},T_{SB})\).  The Darcy map is
\begin{equation}
 \begin{aligned}
 q_x&=x(1-x)N_{q,x}(x,y),\\
 q_y&=(T_{BD})_y+(y-0.9)N_{q,y}(x,y),\\
 p_D&=10[\eta_D N_{pd}(x)+\eta_D(1-\eta_D)N_{D,p}(x,y)],\qquad \eta_D=y/0.9.
 \end{aligned}
 \label{eq:bdf-darcy-map}
\end{equation}
The architecture follows the complete-method configuration specified in Supplement S2. Darcy side walls prescribe zero normal flux, without an additional pressure condition. The learned BD pressure trace sets $\sigma_{B,xy}=\beta(T_{BD})_x$ and $\sigma_{B,yy}=-p_D|_{BD}$, extended linearly through Brinkman by the same principle as the MMS hard map. Hard native traction does not replace the reconstructed-traction audit based on $(u,p)$.

\paragraph{Prespecified complete configuration and correction}
Four validation runs use the prespecified complete configuration. The final BDF physical objective remains fixed throughout three optimizer blocks; only the optimizer schedule changes. Mixed-boundary liftings, loss weights and budgets are adapted to this problem. Seed identifiers, the exact objective and the training schedule are given in Supplements S2 and S4.

The inlet integral and closed side walls give $\bar u_{B,y}^{bc}=-2/3$. The same boundary-mass-consistent rule sets
\begin{equation}
c=-\left[\langle-\operatorname{div}(\sigma_B)_y\rangle+(\mu/K_B)(-2/3)\right],\quad
\delta p_B=c(y-0.9),\quad\delta p_S=0.2c.
\end{equation}
Stress and auxiliaries change consistently as $\delta\sigma=-\delta pI$ and $\delta a_B=(0,c)$; Darcy fields remain unchanged. GL32 coefficients are saved before FEM error evaluation, then checked by GL48 and an independent boundary-stress integral. Independent physical-moment and local-momentum diagnostics are provided in Supplement S4.

\paragraph{Numerical reference and acceptance}

An independently refined, fitted-mesh FEniCS solution supplies the reference. It uses Taylor--Hood velocity/pressure in the upper regions and continuous linear Darcy pressure with the same pseudo-stress convention. Reference refinement, flux reconstruction and its conservation limitations are documented in Supplement S4.

Both raw and corrected states are assessed on two independent field grids against the 18-criterion BDF rule in Supplement S2: six field errors, four native/reconstructed traction vector RMS values, global mass balance and seven hard boundary/kinematic checks. Pressure means are not removed. This tests field accuracy and physical consistency jointly against a checked numerical reference.

\paragraph{Spatial flow and pressure recovery}
\Cref{fig:bdf-spatial} compares a fixed complete-method run with the FEM reference on the fine evaluation grid. The numerical fields reproduce the principal redistribution of parabolic inlet flow through Stokes and Brinkman and laterally within Darcy, together with the layered pressure structure. This spatial agreement complements the four-seed regional error statistics in \cref{fig:bdf-errors}. Displayed pressures are corrected, while velocities and Darcy fields remain unchanged; fixed-run selection and evaluation grids are specified in Supplement S2.

\begin{figure}[p]\centering
\includegraphics[width=\linewidth,height=0.75\textheight,keepaspectratio]{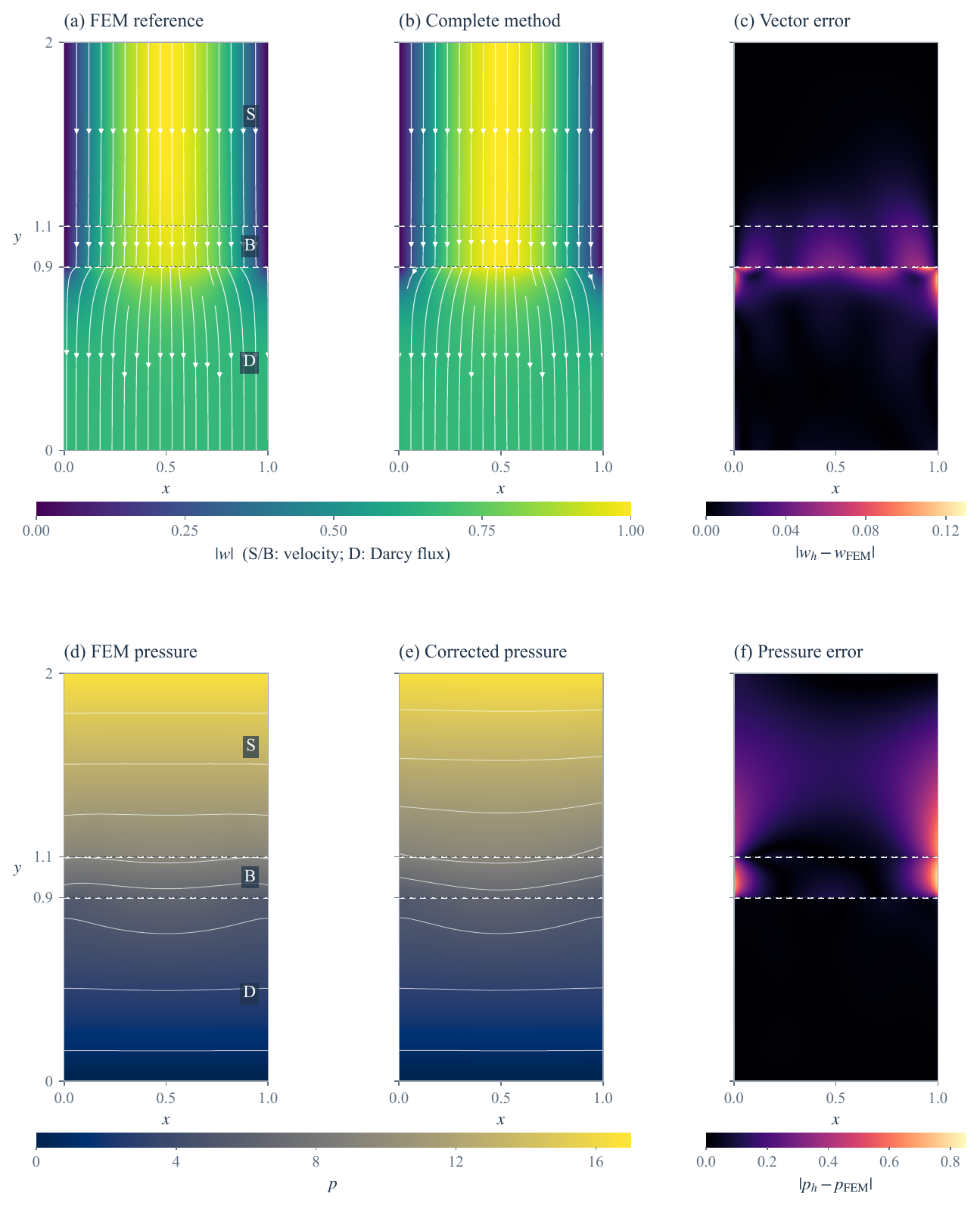}
\caption{BDF spatial validation, $K_B=K_D=0.1$. (a--c) FEM reference, complete-method prediction and vector-error magnitude, with $w=u$ in S/B and $w=q_D$ in D; (d--f) FEM pressure, corrected pressure and absolute error. White curves are regionwise streamlines or pressure contours; dashed lines mark interfaces. Reference/prediction share scales; independent linear error scales cover full sampled ranges. The level-160 FEM solution is a numerical reference.}\label{fig:bdf-spatial}
\end{figure}

\paragraph{Pressure-correction effects and complete-method validation}
All four runs satisfy the 18 acceptance criteria on both evaluation grids before and after correction. The pressure adjustment is small in this case: relative to the level-160 FEM reference, mean Stokes-pressure $L^2$ error changes from 1.4808\% to 1.4504\%, and Brinkman-pressure error from 2.5330\% to 2.5181\% (\cref{fig:bdf-errors}). Both decrease in all four runs. Pressure-specific mesh-refinement changes are reported in Supplement S4; these small gains describe agreement with the numerical reference. The main BDF result is that the mixed-boundary construction meets the field, traction and mass-balance criteria together. Order-48 and invariant checks pass throughout; Supplement S4 gives the unscaled-momentum and other physical diagnostics.

Saved velocities, Darcy pressures and fluxes are elementwise identical before and after correction; BD tractions, SB traction jumps and mass gaps are invariant. Mean training time is $885.06\pm2.07$ s. The four soft-BD comparison corrected checkpoints are separate retrospective evidence in Supplement S4.

\begin{figure}[!htbp]\centering
\includegraphics[width=\linewidth]{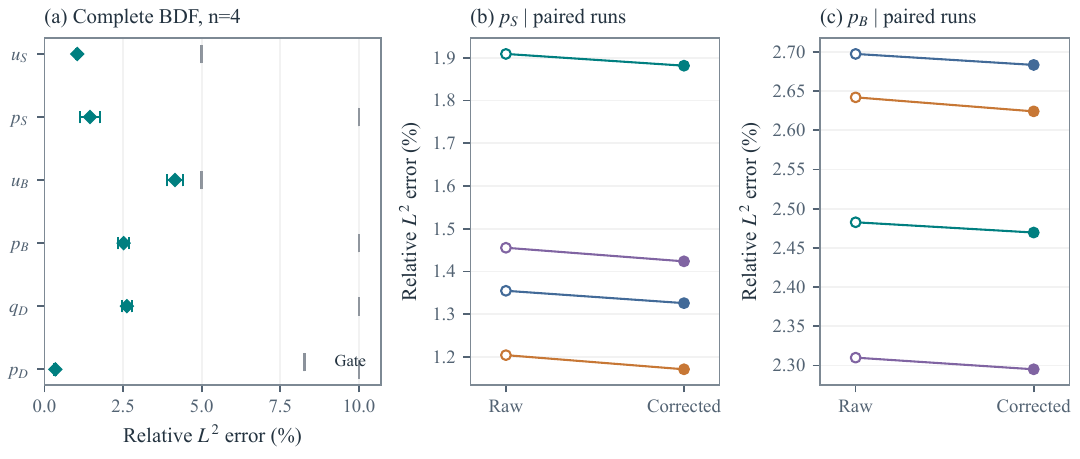}
\caption{Complete-method BDF validation against the independently refined FEM numerical reference. (a) Corrected six-field relative $L^2$ errors: four-seed means and sample SDs; ticks mark error thresholds. (b,c) Stokes/Brinkman pressure errors before and after correction, with each line joining two states of one checkpoint and colors distinguishing runs.}\label{fig:bdf-errors}
\end{figure}

\FloatBarrier
\subsection{Conservation and computational cost}\label{subsec:cost-limits}
For the benchmark kinematic-hard-trace configuration, SB velocity jumps are algebraically zero, while the MMS1 BD mass jump is $(5.77\pm0.21)\times10^{-17}$ and MMS2 is zero at stored precision. These results demonstrate exact control of interface kinematics by shared traces. Volume mass balance and traction reconstructed from $(u,p)$ are evaluated independently; complete records are given in Supplements S2 and S4--S7.

In the recorded CPU batches, mean complete-method training times for nominal MMS1 and the three MMS2 cells are approximately 240--251 s, while pressure correction takes only 0.0506--0.0615 s. Pressure control therefore adds little cost as a post-training step. The complete BDF adaptation has a mean training time of $885.06\pm2.07$ s. These times describe their respective configurations; CPU and GPU batches are not ranked for speed. Matched BD training-time and traction comparisons are in Supplement S5, with benchmark costs in Supplement S2.

\section{Conclusions}\label{sec:conclusion}
The proposed neural residual method solves the steady incompressible Stokes--Brinkman--Darcy equations with a finite-thickness transition layer. Its trial fields satisfy the prescribed exterior and interface kinematic conditions and native BD traction balance exactly. Separate interface traces enforce full velocity continuity at SB and normal-flux continuity at BD, while allowing independent Darcy tangential flux. The mixed state requires only first derivatives of network outputs.

The pressure correction addresses a low-order direction that is weakly penalized by drag-scaled momentum. Boundary mass balance and an integral momentum relation determine its amplitude without reference-pressure labels or retraining. Across the six manufactured-solution cells, every run shows lower Stokes- and Brinkman-pressure $L^2$ errors, full Brinkman-pressure $H^1$ error and normal-derivative error after correction. Corrected upper-pressure $L^2$ means across the six combinations are 0.87\%--2.58\%, with full Brinkman-pressure $H^1$ means of 4.77\%--8.03\%. The paired controls distinguish the roles of exact kinematic sharing, BD traction enforcement and pressure correction.

In the boundary-driven filtration test, all four runs satisfy the 18 prescribed acceptance criteria on both evaluation grids before and after correction. Agreement with the independently refined finite-element reference extends the numerical verification to mixed boundary conditions without manufactured forcing. For this case, pressure changes are small; the main result is the simultaneous recovery of the flow fields and satisfaction of the specified coupling and conservation criteria.

The tests cover steady, two-dimensional rectangular domains with constant coefficients. The additional nonlinear Darcy pressure trace was not recovered accurately at the tested budget, so the present evidence does not establish robust learning of more general interface-pressure traces or universal superiority of pressure-scale separation. The next step is three-dimensional SBD flow: shared traces must be defined on interface surfaces, hard maps must remain compatible where surfaces meet, and pressure-moment constraints must be derived for the resulting geometry. These constructions will be tested before extending the method to variable coefficients and transient flow.

\section*{CRediT authorship contribution statement}
\textbf{Zefeng Liu:} Conceptualization, Methodology, Software, Validation,
Formal analysis, Investigation, Data curation, Writing -- original draft,
Visualization.
\textbf{Hongxing Rui:} Conceptualization, Resources, Writing -- review \& editing, Supervision, Project administration, Funding acquisition.

\section*{Declaration of competing interest}
The authors declare no potential conflicts of interest with respect to the
research, authorship, and/or publication of this article.

\section*{Data availability}
The supplementary material appended to this preprint provides the experimental protocols, seed identifiers and per-run results. The code and supporting numerical data are available from the authors upon reasonable request.

\section*{Acknowledgments}
This work was supported by the National Natural Science Foundation of China
[grant number 12131014].

\section*{Declaration of generative AI and AI-assisted technologies in the manuscript preparation process}
OpenAI ChatGPT/Codex was used to assist with LaTeX formatting and the preparation of submission materials. Responsibility for the scientific content and the final manuscript remains with the authors.

\clearpage
\begin{singlespace}
\small\sloppy
\bibliographystyle{elsarticle-num}
\bibliography{references}
\end{singlespace}

\clearpage
\pdfbookmark[0]{Supplementary Material}{supplementary-material}
\setcounter{section}{0}
\setcounter{subsection}{0}
\setcounter{equation}{0}
\setcounter{table}{0}
\setcounter{figure}{0}
\setcounter{proposition}{0}
\setcounter{assumption}{0}
\setcounter{remark}{0}
\setcounter{algorithm}{0}
\renewcommand{\thesection}{S\arabic{section}}
\renewcommand{\theequation}{S\arabic{equation}}
\renewcommand{\thetable}{S\arabic{table}}
\renewcommand{\thefigure}{S\arabic{figure}}
\renewcommand{\theHsection}{supp.\arabic{section}}
\renewcommand{\theHequation}{supp.\arabic{equation}}
\renewcommand{\theHtable}{supp.\arabic{table}}
\renewcommand{\theHfigure}{supp.\arabic{figure}}
\renewcommand{\theHproposition}{supp.\arabic{proposition}}
\renewcommand{\theHassumption}{supp.\arabic{assumption}}
\renewcommand{\theHremark}{supp.\arabic{remark}}

\begin{center}
{\Large\bfseries Supplementary Material\par}
\medskip
{\large Hard-Trace First-Order Residual Learning for Coupled\\
Stokes--Brinkman--Darcy Flow\par}
\medskip
Zefeng Liu and Hongxing Rui\\
School of Mathematics, Shandong University, Jinan 250100, China
\end{center}

This supplement contains implementation details and conditional analysis (S1), experiment protocols and additional metrics (S2), optimization and sampling diagnostics (S3), BDF validation and soft-BD controls (S4), BD component and nonlinear-trace tests (S5), pressure-moment derivation and interface-release controls (S6), complete-method validation using the method of manufactured solutions (MMS) across six MMS--permeability combinations (S7), and pressure-correction controls for benchmarks and BD pairs (S8). Main-text references refer to the preceding main article.

\section{Supplementary implementation and conditional analysis}
\subsection{Network output components and boundary assignments}
The complete augmented state is defined in main-text Eq.~(13); the component lists below specify the raw network outputs used by its trial maps. For $r\in\{S,B\}$,
\begin{equation}
N_r(x,y;\theta_r)=
\big(
\widehat u_{r,x},\widehat u_{r,y},\widehat p_r,
\widehat\sigma_{r,xx},\widehat\sigma_{r,xy},
\widehat\sigma_{r,yx},\widehat\sigma_{r,yy},
\widehat a_{r,x},\widehat a_{r,y}
\big).
\label{s-eq:upper-raw-output}
\end{equation}
Upper pressure and pressure-gradient auxiliaries are direct outputs. Stress is also direct in the soft-BD control; in the complete method the BD stress components undergo the lifting in main-text Eq.~(34). Velocity uses the main-text hard map. The Darcy network produces

\begin{equation}
N_D(x,y;\theta_D)=
\big(\widehat q_{D,x},\widehat q_{D,y},\widehat p_D\big);
\label{s-eq:darcy-raw-output}
\end{equation}

\begin{table}[!htbp]
\centering
\caption{Boundary data supplied to each hard field construction.  Entries
marked \emph{learned} are trainable traces with fixed compatible endpoints.}
\label{s-tab:hard-map}
\begin{adjustbox}{max width=\linewidth}
\begin{tabular}{lllll}
\toprule
Field & Left & Right & Bottom & Top\\
\midrule
\(u_{\mathrm S}\) & prescribed & prescribed & \(T_{SB}\) (learned) & prescribed\\
\(u_{\mathrm B}\) & prescribed & prescribed & \(T_{BD}\) (learned) & \(T_{SB}\) (shared)\\
\(p_D\) & prescribed & prescribed & prescribed & \(P_{BD}\) (learned)\\
\bottomrule
\end{tabular}
\end{adjustbox}
\end{table}
For the hard MMS maps, exterior diagnostics are
\begin{equation}
\R_{\mathrm S}^{\mathrm{ext}}=u_{\mathrm S}-g_{\mathrm S},\qquad
\R_{\mathrm B}^{\mathrm{ext}}=u_{\mathrm B}-g_{\mathrm B},\qquad
\R_D^{\mathrm{ext}}=\frac{p_D-g_D}{\mu/K_D}.
\label{s-eq:exterior-residuals}
\end{equation}
These boundary residuals vanish by construction. The remaining residuals, their scaling and the continuous and discrete objectives are defined in main-text Section~3.4.
\subsection{Conditional error estimate and its assumptions}
Algebraic admissibility and continuous consistency, including their short proofs, are given in main-text Propositions~1 and 2. The following estimate is conditional on a stability assumption for the chosen scaled operator; it is not an additional proved coercivity theorem. Define
\begin{align}
X_{\mathrm{reg}}={}&
\prod_{r\in\{\mathrm S,\mathrm B\}}
\big[
H^1(\Omega_r)^2\times H^1(\Omega_r)
\times H(\operatorname{div};\Omega_r)^2
\times H(\operatorname{div};\Omega_r)
\big]
\nonumber\\*
&\times H(\operatorname{div};\OD)\times H^1(\OD),
\label{s-eq:regularity-space}
\end{align}
with each pseudo-stress row in \(H(\operatorname{div})\), and additionally require the interface residuals in main-text Eq.~(29) to lie in \(L^2\). Let \(V_0\subset X_{\mathrm{reg}}\) have homogeneous hard traces. Write \(\mathcal A:X_{\mathrm{reg}}\to Y\) for the linear scaled volume/traction operator and \(F\) for its data, with the \(L^2\) product norm on \(Y\) induced by main-text Eq.~(31).

\begin{assumption}[Continuous stability]
\label{s-ass:stability}
For a fixed candidate physical error norm \(\|\cdot\|_X\) on \(V_0\),
there exists a constant \(c_*>0\) such
that
\[
\|\mathcal Av\|_Y\ge c_*\|v\|_X
\qquad\text{for all }v\in V_0.
\]
\end{assumption}

Under \cref{s-ass:stability}, \(\mathcal A(z_\theta-z)=\mathcal A z_\theta-F\) implies

\begin{equation}
\|z_\theta-z\|_X
\le c_*^{-1}\sqrt{\mathcal J(z_\theta)}.
\label{s-eq:conditional-bound}
\end{equation}
Equation~\ref{s-eq:conditional-bound} concerns the continuous residual under the stated stability assumption. Ordinary \(H(\operatorname{div})\) traces lie in \(H^{-1/2}\), hence the additional \(L^2\) requirement. Stability uniform in permeability or layer thickness, and transfer from finite sampling to this bound, remain open. The lifting assumes horizontal rectangular interfaces and compatible corners; curved interfaces or junctions require a different construction.

\section{Experiment protocols and additional metric definitions}
\subsection{Experiment groups and seed identifiers}\label{s-subsec:seed-registry}
The identifiers in \cref{s-tab:seed-registry} initialize network weights and scrambled Sobol sampling. The sample size $n$ counts trained runs within one case, coefficient and configuration; raw and corrected fields are two states of one run. Seeds reused across coefficients do not create additional independent seed replicates. Pairing matches initial raw weights and samples, although a changed hard map can change the initial physical fields.
\begin{table}[!htbp]\centering\footnotesize\setlength{\tabcolsep}{3pt}
\caption{Seed registry. Development runs are excluded from validation statistics.}\label{s-tab:seed-registry}
\begin{adjustbox}{max width=\linewidth}
\begin{tabular}{@{}p{4.1cm}cp{4.8cm}p{4.4cm}@{}}\toprule
Experiment group & $n$ & Seed identifiers & Reuse or pairing\\\midrule
MMS benchmarks & 5 & 20260901--20260905 & Same set for each method/case\\
Complete MMS validation & 3 & 6201, 6227, 6263 & Reused in all six MMS--permeability combinations\\
Hard/soft BD control & 3 & 4301, 4327, 4363 & Paired; reused at three $K_B$\\
Shared/released interfaces & 3 & 6101, 6127, 6163 & Paired; development: 4301\\
Complete BDF validation & 4 & 7201, 7227, 7263, 7291 & Fixed complete configuration\\
Soft-BD BDF control & 4 & 20260912--20260915 & Development: 20260911\\
Coefficient scan & 5 & 3101, 3127, 3163, 3181, 3203 & Separate from hard-BD group\\
Independent hard-BD validation & 5 & 4101, 4127, 4163, 4181, 4203 & Development: 20260909\\
Nonlinear pressure-trace control & 3 & 5101, 5127, 5163 & Paired; development: 20260910\\
\bottomrule\end{tabular}
\end{adjustbox}\end{table}
Single-seed optimization and sampling diagnostics use the first MMS benchmark seed. Numerical identifiers label runs; they do not encode a physical or architectural parameter. They do not ensure bitwise reproducibility across software and hardware versions.
All regional networks have four hidden layers of width 64; MMS hard-trace and complete BDF configurations use three trace networks with three hidden layers of width 32. Activation is $\tanh$, initialization is Xavier-normal, inputs are physical coordinates and arithmetic is float64. Configuration parameter counts are in main-text Table~2; the earlier soft-BD BDF architecture has 43,865 parameters.

\subsection{Compared formulations}
\label{s-subsec:compared-formulations}

The experiment compares the following three implementations.

\begin{enumerate}[leftmargin=1.8em]
\item \textbf{Conventional strong-form PINN.}
The Stokes and Brinkman networks each output \((u_x,u_y,p)\), whereas the
Darcy network outputs one raw scalar whose physical pressure is multiplied by
\(\mu/K_D\).  Pseudo-stress and Darcy flux are reconstructed as
\[
\sigma_r=\mu_r\nabla u_r-p_r I,
\qquad
q_D=-\frac{K_D}{\mu}\nabla p_D.
\]
The momentum equations and Darcy mass equation therefore require second
derivatives of neural outputs.  All exterior conditions and all four interface
conditions in main-text Eqs.~(9) and (10) are imposed by unscaled, unit-weight soft
mean-square penalties.  The reconstructed constitutive identities are not
independent trainable equations.

\item \textbf{Soft first-order residual model.}
The upper networks independently predict \((u,p,\sigma)\), and the Darcy
network predicts \((q_D,p_D)\), with the raw Darcy pressure again multiplied
by \(\mu/K_D\).  This removes second derivatives, but every exterior and
interface condition remains a soft penalty.  This baseline has no pressure
auxiliaries \(a_r\), no pressure-Poisson residuals, no trace networks, no
Coons lifting, and no separate unit-scale Darcy interface-pressure trace.  Its
volume and interface residual groups are used in their unscaled base forms,
with unit coefficients after the componentwise group mean squares are formed.
The exterior Darcy-pressure residual is also unscaled for both baselines.

\item \textbf{Kinematic hard-trace first-order model.}
This is the original soft-BD configuration in main-text Section~3.  It predicts independent first-order
stress and flux variables, adds \(a_{\mathrm S}\) and \(a_{\mathrm B}\), and
uses the compatible traces \(T_{SB}\), \(T_{BD}\), and \(P_{BD}\).  The
exterior velocity and Darcy-pressure data, Stokes--Brinkman velocity
continuity, and Brinkman--Darcy normal mass continuity are identities of the
trial map.  Only the first-order volume equations and the two traction groups
remain active in the optimization.  Brinkman momentum is divided by
\(1+\mu/K_B=101\), and only the tangential component of the
Brinkman--Darcy traction residual is divided by
\(1+\lambda_{\mathrm{BJS}}=2\); the normal pressure--traction component is not
attenuated.  The Darcy constitutive residual
\(q_D+(K_D/\mu)\nabla p_D\) is unscaled in the reported model; in particular,
no additional \(K_D^{-1/2}\) factor is used.
\end{enumerate}

Parameter counts and the nominal-budget comparison scope are given in main-text Table~2 and the experimental settings.

\subsection{Collocation sampling, seeds, and reproducibility}
\label{s-subsec:training-samples}

Each training run uses fixed, full-batch collocation sets.  The points are
generated once at the start of a case--seed run and are not resampled during
Adam or L-BFGS.  All three methods receive the same point sets within a given
case--seed pair.  Sobol points for the two-dimensional interiors and one-dimensional boundaries and interfaces are generated with PyTorch\textquotesingle s \texttt{torch.quasirandom.}\allowbreak\texttt{SobolEngine}, using \texttt{scramble=True} and the run-specific seed. Bratley and Fox provide the classical sequence-generator reference \cite{BratleyFox1988}; the recorded PyTorch versions are listed by batch in \cref{s-subsec:hardware-environment}.  The counts are summarized in
\cref{s-tab:training-points}.

\begin{table}[!htbp]
\centering
\caption{Fixed collocation counts for one training run.  ``Per region'' means
that the stated count is used separately in \(\OS\), \(\OB\), and \(\OD\).}
\label{s-tab:training-points}
\small
\begin{adjustbox}{max width=\linewidth}
\begin{tabular}{
  >{\raggedright\arraybackslash}p{3.7cm}
  >{\raggedright\arraybackslash}p{2.5cm}
  >{\raggedright\arraybackslash}p{7.8cm}}
\toprule
Point group & Count & Allocation\\
\midrule
Interior & 512 per region & two-dimensional scrambled Sobol\\
Exterior boundary & 192 per region &
64 per edge in Stokes/Darcy; 96 per vertical Brinkman side\\
Stokes--Brinkman interface & 192 & one-dimensional scrambled Sobol\\
Brinkman--Darcy interface & 192 & one-dimensional scrambled Sobol\\
\bottomrule
\end{tabular}
\end{adjustbox}
\end{table}

For the kinematic hard-trace model, exterior residuals remain in the sum but vanish
algebraically; \(\R_{SB}^{u}\) and \(\R_{BD}^{m}\) are computed as diagnostics
and excluded from the sum.  For the two soft formulations, the same
sampled conditions contribute gradients to the objective.

The five benchmark seeds are listed in \cref{s-tab:seed-registry}.

For each run, the seed initializes the Python, NumPy, and PyTorch random-number
generators and the scrambling of the Sobol engines.  It therefore controls
both the Xavier initialization and the sampled training sets.  This procedure
does not assert bitwise reproducibility across different CUDA, driver, or
PyTorch versions; the seed-to-seed statistics quantify the observed run
variation on the stated platform.

\subsection{Optimization budget and hardware}
\label{s-subsec:optimization-budget}

Every model is optimized first with Adam at learning rate \(10^{-3}\)
\cite{KingmaBa2015}.  Exactly 400 completed Adam parameter updates are made.
The Adam endpoint initializes full-batch PyTorch L-BFGS optimization using
learning rate \(0.8\), history size \(50\), strong-Wolfe line search, gradient
tolerance \(10^{-10}\), and change tolerance \(10^{-12}\)
\cite{LiuNocedal1989}.  For an individual call, if \(m\) denotes
\texttt{max\_iter}, then
\texttt{max\_eval} is
\(\max\{\lfloor5m/4\rfloor,m+1\}\).  The nominal totals are listed in \cref{s-tab:optimizer-budgets}.
\begin{table}[!htbp]
\centering
\caption{Nominal optimization totals for all three methods.  The parenthetic
number is the corresponding maximum total closure-evaluation allowance.}
\label{s-tab:optimizer-budgets}
\begin{adjustbox}{max width=\linewidth}
\begin{tabular}{lcc}
\toprule
Case & Completed Adam updates &
L-BFGS \texttt{max\_iter} (\texttt{max\_eval})\\
\midrule
MMS1 & 400 & 300 (375)\\
MMS2 & 400 & 600 total (750 total)\\
\bottomrule
\end{tabular}
\end{adjustbox}
\end{table}
The L-BFGS values are nominal upper bounds.  They are not statements that every
run accepted exactly 300 or 600 quasi-Newton iterations.  Moreover, one
accepted L-BFGS step may call the loss closure several times during line
search.  Closure evaluations are therefore recorded separately and are the
horizontal coordinate whenever an optimization plot uses that count.

The MMS2 total was selected after a preliminary common 400+300 diagnostic.
PINN and soft first-order models use one L-BFGS call with
\texttt{max\_iter}=600; every archived hard model uses 300 followed by a new
300-iteration optimizer. All seeds receive the same total, but the hard
history reset makes the optimizer paths different. This is a post-pilot
nominal-budget comparison.

All MMS benchmark training runs were executed in float64 on the same NVIDIA RTX
3050 Ti GPU.  The kinematic hard-trace model is larger and has more differentiated
quantities, so an equal nominal optimizer budget is not an equal-runtime or
equal-FLOP comparison.  The independent frozen-checkpoint evaluation below was
executed in float64 on the CPU.

\subsection{Additional norms and quadrature conventions}\label{s-subsec:supp-evaluation}
Relative $L^1/L^2$ and full $H^1$ errors, and the within-checkpoint regional maxima, are defined in the main-text evaluation section. The same quadrature notation gives the sampled maximum error
\begin{equation}
E_{L_h^\infty,r}(v)
=\frac{\max_{x_i\in X_r}|e_v(x_i)|}
{\max_{x_i\in X_r}|v^{\star}(x_i)|}.
\label{s-eq:relative-linf-error}
\end{equation}
The Darcy-flux norm is
\begin{equation}
E_{H(\operatorname{div}),D}(q_D)
=\frac{\left(
Q_D[|q_{D,\theta}-q_D^{\star}|^2]
+Q_D[|\nabla\!\cdot(q_{D,\theta}-q_D^{\star})|^2]
\right)^{1/2}}
{\left(
Q_D[|q_D^{\star}|^2]
+Q_D[|\nabla\!\cdot q_D^{\star}|^2]
\right)^{1/2}}.
\label{s-eq:relative-hdiv-error}
\end{equation}
The sampled maximum is not a continuum supremum. Absolute norms are also retained; a relative value is undefined when the reference norm is at most $10^{-12}$. Pressure means are not removed. Tensor-product trapezoidal quadrature uses Euclidean vector and Frobenius gradient magnitudes. Metrics are accumulated over the entire evaluation grid before aggregation, including when differentiation is batched in groups of at most 4096 points. Conservation diagnostics are absolute $L^2$ norms of Stokes/Brinkman divergence and Darcy mass residual.
\subsection{Cross-seed aggregation}
For any checkpoint metric $M_s$, the mean and sample standard deviation over the $n$ runs in the relevant group are
\begin{equation}
\overline M=\frac{1}{n}\sum_{s=1}^{n}M_s,\qquad
s_M=\left[\frac{1}{n-1}\sum_{s=1}^{n}(M_s-\overline M)^2\right]^{1/2}.
\label{s-eq:seed-statistics}
\end{equation}
Each grid is reduced to one checkpoint metric first. Grid nodes and raw/corrected states are not independent replicates. Pass counts count individual checkpoints satisfying all applicable acceptance criteria; thresholds are not applied to group means. Standard deviations describe run-to-run variability and are not confidence intervals. Sample sizes and pairing are specified in \cref{s-tab:seed-registry}. In the tables, Y/N denote pass/fail.
\subsection{Evaluation families and diagnostic conventions}\label{s-subsec:gate-families}
The benchmark acceptance thresholds are defined in main-text Table~3. The following rules apply to the other experiment groups; all inequalities are strict and all applicable acceptance criteria must pass within each checkpoint before seeds are aggregated.
\begin{table}[!htbp]\centering\small\setlength{\tabcolsep}{4pt}
\caption{Additional acceptance families. Field thresholds are relative $L^2$ ratios; traction thresholds use unscaled physical quantities.}\label{s-tab:additional-gates}
\begin{adjustbox}{max width=\linewidth}
\begin{tabular}{@{}p{2.3cm}p{8.8cm}c@{}}\toprule
Family & Quantities and strict upper bounds & Gates\\\midrule
MMS component & Two upper velocities: 0.05; three pressures and Darcy flux: 0.10 & 6\\
 & Native/reconstructed SB/BD traction, each tangential and normal component RMS: 0.10 & 8\\\midrule
BDF joint & Same six relative field bounds as above & 6\\
 & Native/reconstructed SB/BD traction, each vector RMS: 0.10 & 4\\
 & Global mass defect relative to inlet flux: 0.01 & 1\\
 & Each hard boundary/kinematic maximum absolute error: $10^{-12}$ & 7\\
\bottomrule\end{tabular}
\end{adjustbox}\end{table}
The seven BDF hard checks are top Stokes velocity, side Stokes velocity, side Brinkman velocity, side Darcy normal flux, bottom Darcy pressure, SB velocity continuity and BD normal-flux continuity. The MMS family totals 14 criteria and the BDF family 18. The 14-gate diagnostic is also applied retrospectively to the MMS benchmark checkpoints on the extended grids; it does not replace their original benchmark rule. For nonlinear-trace tests it is a diagnostic screen, not an application-validated accuracy requirement. Full $H^1$, pressure-derivative error and raw momentum are additional diagnostics, not extra criteria in these totals.
\begin{table}[!htbp]\centering\small\setlength{\tabcolsep}{4pt}
\caption{Evaluation grids. Regional entries are horizontal by vertical; all are independent of training samples.}\label{s-tab:evaluation-grids}
\begin{adjustbox}{max width=\linewidth}
\begin{tabular}{@{}p{3.3cm}p{3.7cm}p{3.7cm}p{2.5cm}@{}}\toprule
Group & Primary regional grid & Fine regional grid & Interface points\\\midrule
MMS benchmark & $81\times61$ each & --- & 401\\
Complete/component MMS & $161\times121$ each & $321\times241$ each & 801 / 1601\\
BDF fields & S,D: $61\times55$; B: $61\times14$ & S,D: $121\times241$; B: $121\times61$ & See text\\
\bottomrule\end{tabular}
\end{adjustbox}\end{table}
Complete BDF checks traction on 401/1601 main/fine interface points. The earlier soft-BD protocol uses a 401-point traction check plus a denser six-field check; its retrospective correction diagnostics retain their specified evaluation sets. A fine grid checks quadrature sensitivity of a fixed network, not convergence with training resolution.

\emph{Native} traction uses the independently learned stress; \emph{reconstructed} traction uses $\sigma_r^{\rm rec}=\mu_r\nabla u_r-p_rI$. For a vector residual with $m$ components, vector RMS is $\sqrt{Q[|R|_2^2]/|G|}$, whereas component-averaged RMS is smaller by $\sqrt m$. Brinkman momentum tables use the latter. ``Maximum component'' means $\max_{i,j}|R_j(x_i)|$; ``maximum norm'' means $\max_i|R(x_i)|_2$. A reconstructed momentum residual uses the reconstructed stress. Grid diagnostics use the stated evaluation grid; values from the separate soft-BD evaluator are not treated as the same sample set. Percentages are explicitly marked; other residual values are nondimensional, physical and unscaled.
\subsection{Baseline field and conservation displays}
\Cref{s-fig:cross-seed-errors} supplements the main $L^2$ comparison with full derivative norms; \cref{s-fig:physics-audit} reports independent conservation and reconstructed-traction diagnostics.

\begin{figure}[!htbp]\centering
\includegraphics[width=\linewidth]{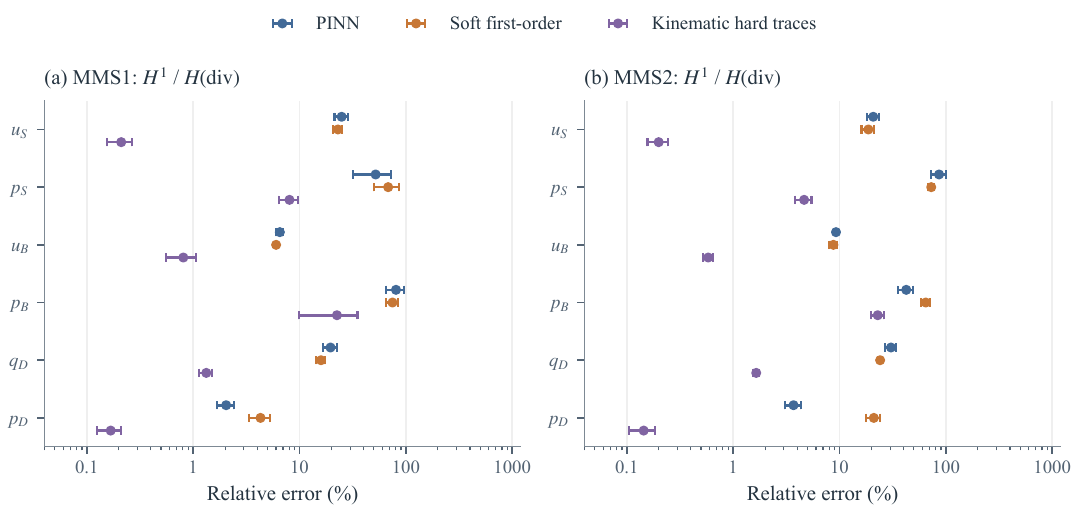}
\caption{Full $H^1$ relative errors of five-seed MMS benchmarks, using $H(\mathrm{div})$ for Darcy flux: (a) MMS1; (b) MMS2. Points and intervals show means and sample SDs in percent, with the main accuracy figure\textquotesingle s configuration colors.}\label{s-fig:cross-seed-errors}
\end{figure}

\begin{figure}[!htbp]\centering
\includegraphics[width=\linewidth]{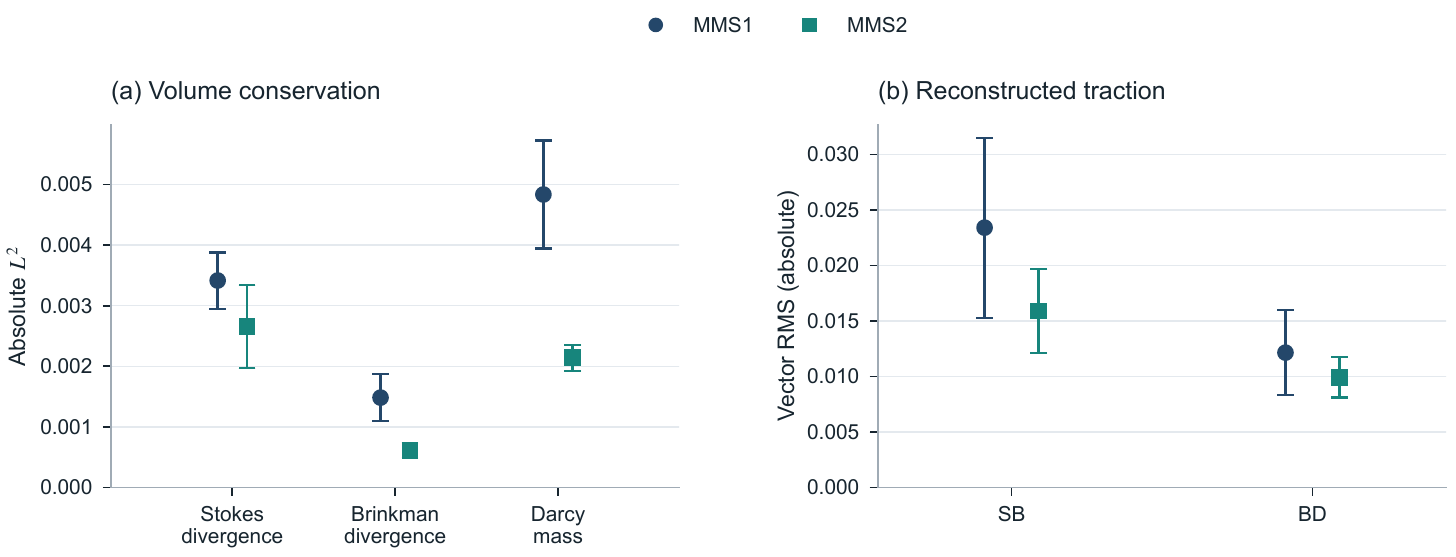}
\caption{Conservation and reconstructed traction of benchmark kinematic hard traces. Divergence/mass use absolute $L^2$ and traction uses vector RMS; means and sample SDs are shown on separate scales.}\label{s-fig:physics-audit}
\end{figure}

\subsection{Computational environment and timing scope}\label{s-subsec:hardware-environment}
The table collects recorded environments; CPU models and software versions absent from the saved records are not inferred. Timing comparisons use the same hardware and corresponding batch, with no CPU/GPU cross-batch speed ranking.
\begin{center}\small
\begin{adjustbox}{max width=\linewidth}
\begin{tabular}{@{}p{0.25\linewidth}p{0.65\linewidth}@{}}\toprule
Experimental batch & Recorded environment and timing scope\\\midrule
MMS benchmarks & RTX 3050 Ti GPU; float64; times include the training routine's final evaluation.\\
Complete MMS and component controls & CPU, float64; two threads per process, at most two concurrent processes. The MMS1 low-permeability protocol records PyTorch 2.7.1+cu118.\\
Earlier soft-BD BDF & RTX 3050 Ti Laptop GPU; PyTorch 2.7.1+cu118; serial development, at most two concurrent validation processes.\\
Complete BDF validation & CPU, float64; PyTorch 2.7.1+cu118; two threads per process; complete configuration in S4.\\
FEM reference & FEniCS; Taylor--Hood and continuous linear Darcy pressure; GMRES interface iteration and MUMPS subproblem solves.\\\bottomrule
\end{tabular}
\end{adjustbox}
\end{center}

\begin{figure}[!htbp]\centering
\includegraphics[width=\linewidth]{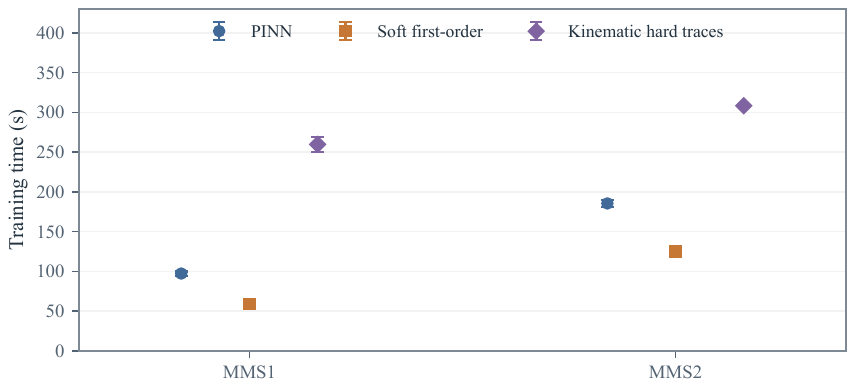}
\caption{GPU benchmark time including final evaluation: five-seed means and sample SDs for each method and MMS case.}\label{s-fig:baseline-cost}
\end{figure}

\subsection{Spatial-figure protocol}
The main flow, pressure-contour and interface figures use the complete-method checkpoint with seed 6201 at $K_B=K_D=10^{-2}$ for each MMS. This is a fixed representative run, not a minimum-error selection. The three-cut figure uses all three validation seeds 6201, 6227 and 6263. The MMS2 permeability maps use seed 6201 at each of the three $K_B$ values, keeping $K_D=10^{-2}$. Field panels use the saved $161\times121$ grid in each region; the SB pressure and BD jump diagnostics use the saved 801-point interface traces, while displayed BD normal-flow curves use the boundary rows of the regional field grids. No model is evaluated for figure generation. Flow magnitudes and vector errors use $w=u$ in S/B and $w=q_D$ in D, with no interpolation across material interfaces. Raw/corrected pressures retain the original gauge, and signed errors use the saved exact fields. Means and sample SDs use the stated group sizes; profile bands are seed minima and maxima. The paired BD and interface-release profile figures in S5 and S6 concern different component-control cohorts and do not repeat the complete-method validation cuts.

\paragraph{Pressure-direction diagnostic}
main-text Figure~5(c,d) uses the fixed complete-method MMS2 checkpoint with seed 6201, $K_B=10^{-6}$ and $K_D=10^{-2}$, not an error-selected run. The stored coefficient is $c_*=0.7066119935322748$. Independent order-48 physical moments are $-0.7066119935322743$ before correction and $4.44\times10^{-16}$ afterward. Panel (c) uses the exact affine identity between these checked endpoints; panel (d) evaluates $p_S+c h_B$ and $p_B+c(y-y_{BD})$ on the frozen main grid with its stored quadrature weights and pressure gauge. Exact pressure is used only to evaluate errors, not to choose the coefficient. Panels (a,b) are analytical diagnostics; the sensitivity curve starts from an exact solution and is not a trained-network loss or a new parameter-sweep experiment. In this illustrative run the unscaled Brinkman momentum component RMS slightly increases, $691.8125\to692.1413$; moment cancellation does not imply pointwise momentum improvement.

\paragraph{BDF spatial comparison}
main-text Figure~11 uses seed 7201, the first prespecified complete-BDF validation run, and its corrected fine-grid fields. The level-160 FEM reference is sampled at 121 horizontal points and 241/61/241 vertical points in S/B/D. This is denser sampling of the same FEM reference, not a new FEM solve. Reference/predicted fields share scales, and each absolute-error scale covers the full sampled range without clipping. Streamlines and pressure contours are drawn separately in each region, without cross-interface smoothing. The original gauge and Darcy-field invariance are retained.

\section{Optimization, sampling, and evaluation-grid diagnostics}
These archived single-seed diagnostics are separate from the main cross-seed evaluation. All diagnostic field grids are independent of training Sobol points.
\subsection{Optimization endpoints}
Representative seed 20260901 uses a $31\times21$ equal-weight grid per region. MMS1 uses 400 Adam and 300 L-BFGS iterations; MMS2 uses 400 and 600. All six hard-trace field errors improve between L-BFGS endpoints, by factors 1.72--10.54 and 2.43--12.50, respectively. Adam is evaluated every 20 updates; L-BFGS has only two independently audited endpoints. Line-search closures are not accepted iterations, and monotonic improvement is not claimed. Main results use independent trapezoidal evaluation.
\subsection{Sampling sensitivity}
The same seed tests interior/exterior/interface point-count triples $(128,48,48)$, $(512,192,192)$ and $(1024,384,384)$, with 400+300 for MMS1 and 400+600 for MMS2. Original diagnostics use $80\times60$ equal-weight field grids and 401 equal-weight interface points. All settings meet their stored diagnostic thresholds, but pressure and Darcy-flux errors are nonmonotone in sample count. This is a one-seed, three-level sensitivity check, not a convergence-order or uncertainty study.
\subsection{Fixed-checkpoint evaluation grids}
One representative checkpoint per method and case is reevaluated on $80\times60$, $160\times120$ and $320\times240$ equal-weight grids. Hard-model maximum velocity/pressure errors change by less than 0.6\%; Darcy-flux errors change by about 2.5\%--3.6\% (0.03--0.06 percentage points), without changing field-gate decisions. The MMS1 hard checkpoint is a separate diagnostic model; MMS2 uses formal seed 20260901. This is not quadrature certification of all thirty checkpoints.

\section{BDF validation, reference and soft-BD controls}
This section presents complete-method validation first, followed by the earlier soft-BD protocol and its separate correction control. Both use the physical BDF problem in main-text Section~5.5 and the independently refined reference described here. Different architectures and training paths prevent interpretation as an isolated hard-BD ablation.
\subsection{Complete BDF: verification and adverse changes}
Both states pass all 18 joint acceptance criteria on both grids in 4/4 seeds; the 14 MMS-style criteria also pass. GL48 and independent boundary-stress integration differ by at most $6.22\times10^{-15}$. Coefficients for seeds 7201, 7227, 7263 and 7291 are 0.02312238, 0.02811921, 0.01977362 and 0.02392622. The original final BDF physical loss remains fixed; the main text defines the method; training budgets are given below. Velocities and Darcy fields are elementwise identical; traction, boundary, kinematic, constitutive and auxiliary invariant checks pass 4/4.

\begin{table}[H]\centering\normalsize\renewcommand{\arraystretch}{1.18}\setlength{\tabcolsep}{4pt}
\caption{Complete-BDF additional momentum and fine-grid diagnostics; the final column counts increases after correction.}
\begin{adjustbox}{max width=\linewidth}
\begin{tabular}{@{}lcccc@{}}\toprule
Grid & Metric & Raw & Corrected & Increases\\\midrule
main & B momentum RMS & \(0.0608\pm0.0036\) & \(0.0631\pm0.0039\) & 4/4\\
main & B recon. mom. RMS & \(3.1381\pm0.1124\) & \(3.1436\pm0.1132\) & 4/4\\
main & B momentum maximum component & \(0.3181\pm0.0886\) & \(0.3418\pm0.0894\) & 4/4\\
fine & $p_S$: $L^2$ (\%) & \(1.4800\pm0.3036\) & \(1.4497\pm0.3056\) & 0/4\\
fine & $p_B$: $L^2$ (\%) & \(2.5247\pm0.1750\) & \(2.5097\pm0.1745\) & 0/4\\
fine & B momentum RMS & \(0.0605\pm0.0036\) & \(0.0628\pm0.0038\) & 4/4\\
fine & B recon. mom. RMS & \(3.1357\pm0.1124\) & \(3.1411\pm0.1131\) & 4/4\\
fine & B momentum maximum component & \(0.3181\pm0.0886\) & \(0.3418\pm0.0894\) & 4/4\\
\bottomrule\end{tabular}
\end{adjustbox}\end{table}

\subsection{Complete-method training protocol}
Seeds 7201, 7227, 7263 and 7291 were specified before training, and all four runs were included in the analysis. Runs use CPU float64, two threads per process, 512 interior and 192 exterior points per region, and 256 points per interface. The three blocks use 600 Adam($10^{-3}$)+1000 L-BFGS, 300 Adam($10^{-4}$)+1000 L-BFGS, and 300 Adam($10^{-4}$)+1000 L-BFGS, resetting optimizers between blocks. L-BFGS uses rate 0.8, history 50, max\_eval 1250 and strong-Wolfe search. The original final BDF physical objective is fixed throughout; blocks change only the optimizer schedule, not residual components. BDF retains the first-order framework, hard-interface principles and correction rule, with liftings, loss weights and budgets adapted to the mixed boundaries.

The complete configuration uses \cref{s-eq:bdf-loss} from the first update through all three optimizer blocks. The staged changes below belong only to the earlier soft-BD control.

\subsection{Soft-BD staged objective and validation protocol}

Seed 20260911 was used for development. The three-stage procedure below
was fixed before reserved seeds 20260912--20260915 were run. All stages use
the same architecture and physical data, with progressively changed
constitutive and traction weights.

For reproducibility, define \(\langle R\rangle_2\) as the empirical mean
of squared residual components over the fixed points, and set
\begin{equation}
 \begin{aligned}
 C_r&=\sigma_r-(\nabla u_r-p_rI),&
 G_r&=a_r-\nabla p_r,& P_r&=\operatorname{div}a_r,\\
 M_S&=-\operatorname{div}\sigma_S,&
 M_B&=-\operatorname{div}\sigma_B+10u_B,&
 Q&=q_D+0.1\nabla p_D.
 \end{aligned}
 \label{s-eq:bdf-residuals}
\end{equation}
Let \(F=\operatorname{mean}_{\mathrm{bottom}}(q_y)+2/3\), using the
fixed bottom Sobol points, and let \(R_j^n,R_j^r\) denote the native
and reconstructed physical traction residuals at \(j=SB,BD\).
The final-stage objective is
\begin{equation}
 \begin{aligned}
 \mathcal L_{BDF}={}&\sum_{r=S,B}\Big[
    \langle C_r\rangle_2+20\langle\operatorname{div}u_r\rangle_2
    +\langle G_r/10\rangle_2+\langle P_r/10\rangle_2\Big]\\
 &+\langle M_S/10\rangle_2+100\langle M_B/21\rangle_2
    +50\langle Q\rangle_2+20\langle\operatorname{div}q_D\rangle_2\\
 &+20F^2+5\sum_{j=SB,BD}
       \big(\langle R_j^n\rangle_2+\langle R_j^r\rangle_2\big).
 \end{aligned}
 \label{s-eq:bdf-loss}
\end{equation}
The denominator 21 is the inherited Brinkman scale \(1+10+10\).
Stage 1 uses the same volume terms except that each constitutive term is
\(\langle C_r/10\rangle_2\), and includes only native tractions:
\(R_{SB}^n/10\) and
\(((R_{BD}^n)_x/(1+\beta),(R_{BD}^n)_y/10)\), each with weight 5.
Stage 2 restores the constitutive terms and adds physical reconstructed
traction penalties of weight 5 while retaining those native scalings.
Stage 3 is exactly \cref{s-eq:bdf-loss}.

Reconstruction uses \(\sigma^r=\nabla u-pI\), requiring only first
spatial derivatives of the primary network outputs.  These added terms
vanish for a smooth exact solution.  Moreover, for the same point set,
\begin{equation}
 \langle A\rangle_2+\langle B\rangle_2
 =2\langle(A+B)/2\rangle_2+\tfrac12\langle A-B\rangle_2.
 \label{s-eq:bdf-dual-traction}
\end{equation}
Here \(A-B=(C_S-C_B)n\) at SB and \(A-B=C_Bn\) at BD when
\(A=R^n\), \(B=R^r\).  The penalty therefore checks the corresponding
interface constitutive components directly.  This is an algebraic identity,
not a trace bound derived from an interior constitutive \(L^2\) norm or a
parameter-uniform stability result.

The entire three-stage procedure was specified in advance for reserved
seeds 20260912--20260915.  Each seed starts from its own random initialization,
then continues only its own preceding checkpoint.  The schedule is shown in \cref{s-tab:bdf-budget}.

\begin{table}[!htbp]
 \centering\small
 \caption{Soft-BD BDF schedule for every reserved seed.  Each stage starts a
 new optimizer history; all three stages count toward the reported cost.}
 \label{s-tab:bdf-budget}
 \begin{adjustbox}{max width=\linewidth}
\begin{tabular}{@{}lrrrl@{}}
 \toprule
 Stage & Adam updates & Adam rate & L-BFGS limit & Traction objective\\
 \midrule
 1 & 600 & \(10^{-3}\) & 1000 & Scaled native\\
 2 & 300 & \(10^{-4}\) & 1000 & Scaled native + physical reconstructed\\
 3 & 300 & \(10^{-4}\) & 1000 & Both physical\\
 \bottomrule
 \end{tabular}
\end{adjustbox}
\end{table}

All stages use float64, physical input coordinates, fixed scrambled Sobol
points (512 interior points per region, 192 boundary samples per region, and
256 points per interface), and the same network widths.  L-BFGS uses rate
0.8, history size 50, strong-Wolfe line search, and gradient/change tolerances
\(10^{-10}/10^{-12}\).  The total nominal budget is 1200 Adam updates
plus 3000 L-BFGS iterations.  Development was serial and validation used at
most two concurrent processes on the RTX 3050 Ti Laptop GPU with PyTorch
2.7.1+cu118.  Concurrent-run wall times are not a controlled speed comparison.
Reference fields are evaluated only after a stage finishes; they are not
used in the loss, for point selection, or to alter a reserved seed's schedule.

\subsection{Independent finite-element reference and acceptance rule}

An independent FEniCS reference uses Taylor--Hood velocity/pressure in the
upper regions and continuous piecewise-linear Darcy pressure, coupled by
GMRES interface iteration (tolerance \(10^{-10}\)) with MUMPS subproblem
solves on fitted meshes.  Its upper bilinear form uses
\(\nabla u:\nabla v\), consistently with the nonsymmetric pseudo-stress.
Four mesh levels were computed without neural input; \cref{s-tab:bdf-fem}
reports reference conservation diagnostics.  At level 160, the interface
iteration residual was \(1.78\times10^{-13}\).  The largest equal-weight discrete relative \(L^2\)
change on the reference sample grids among the six fields from level 80 to 160 was \(0.201\%\), for
Darcy flux. The corresponding changes for $p_S$ and $p_B$ were $0.0089275\%$ and $0.0337062\%$, respectively, as recorded in the \texttt{refinement\_160} entry under \texttt{FEM\_refinement} in \nolinkurl{data/bdf_validation_results.json}. These adjacent-mesh changes are not rigorous bounds on error relative to the continuous solution; the small BDF pressure-error reductions reported here are relative to the level-160 reference. The reference met the criteria of field changes below 1\%,
global mass defect below 1.5\%, interface integrated flux jump below 0.1\%,
and iteration residual below \(10^{-10}\).

\begin{table}[!htbp]
 \centering\small
 \caption{BDF finite-element reference refinement.  Conservation values are
 percentages; they are properties of the numerical reference, not exact truth.}
 \label{s-tab:bdf-fem}
 \begin{adjustbox}{max width=\linewidth}
\begin{tabular}{@{}rrrr@{}}
 \toprule
 Level & Total degrees of freedom & Global mass defect & Interface flux jump\\
 \midrule
 20 & 4,572 & 4.0467 & 0.3361\\
 40 & 17,780 & 2.2461 & 0.0859\\
 80 & 70,116 & 1.1823 & 0.0218\\
 160 & 278,468 & 0.6066 & 0.00552\\
 \bottomrule
 \end{tabular}
\end{adjustbox}
\end{table}

The raw Darcy flux reconstructed from continuous pressure is not strictly
\(H(\operatorname{div})\)-conservative, and the sampled reference flux is
also projected into a continuous vector space.  Side-wall reconstruction
leakage and interface projection error are retained in the reference audit;
the level-80-to-160 pointwise interface flux change is about 1.066\%.
Thus field errors are measured relative to a converged numerical reference,
and FEM and neural conservation numbers should not be treated as identical
discretization measures.

Both BDF groups use the reference-based field norms and BDF joint rule in \cref{s-tab:additional-gates,s-tab:evaluation-grids}. The neural global defect is the absolute top-plus-bottom outward flux normalized by the analytic inlet magnitude $2/3$. Side walls are impermeable by construction. The earlier soft-BD dense check covers the six fields separately from its 401-point traction evaluation. Development runs are excluded from validation statistics.
\begin{table}[!htbp]
 \centering\small\setlength{\tabcolsep}{4pt}
 \caption{Soft-BD BDF primary errors and mass defect in percent.  Dev is the
 development seed; V1--V4 are reserved seeds.  Every row passes the complete
 rule, including both traction definitions and the dense-grid field check.}
 \label{s-tab:bdf-seeds}
 \begin{adjustbox}{max width=\linewidth}
\begin{tabular}{@{}llrrrrrrr@{}}
 \toprule
 Seed & Role & \(u_S\) & \(p_S\) & \(u_B\) & \(p_B\) & \(q_D\) & \(p_D\) & Mass\\
 \midrule
20260911 & Dev & 1.047 & 2.400 & 4.516 & 2.562 & 2.570 & 0.326 & 0.205\\
20260912 & V1 & 0.829 & 2.462 & 3.663 & 2.640 & 2.252 & 0.324 & 0.200\\
20260913 & V2 & 0.980 & 2.592 & 4.094 & 2.360 & 2.423 & 0.333 & 0.267\\
20260914 & V3 & 1.040 & 1.994 & 4.224 & 1.980 & 2.524 & 0.319 & 0.220\\
20260915 & V4 & 0.999 & 2.483 & 4.181 & 2.546 & 2.474 & 0.386 & 0.254\\
 \bottomrule
 \end{tabular}
\end{adjustbox}
\end{table}

\subsection{Soft-BD BDF configuration: retrospective correction only}
Seeds 20260912--20260915 retain 43,865 parameters, hard kinematics, soft BD and their original staged training; correction adds no training. The same BDF correction uses GL32 with GL48/boundary-integral checks, saving coefficients before reference errors; 93 checks pass. Main-grid S-pressure $L^2$ (mean$\pm$sample SD) changes from $2.3832\pm0.2653\%$ to $2.3480\pm0.2638\%$ and B-pressure from $2.3815\pm0.2920\%$ to $2.3677\pm0.2916\%$, both modestly improving in 4/4 seeds. Fine-grid means change as $2.3823\%\to2.3472\%$ and $2.3723\%\to2.3584\%$. B-momentum component RMS on the soft-BD evaluator rises from 0.06817 to 0.07049; it is not the same point set as the complete BDF grid RMS. Velocities, Darcy fields and tractions remain unchanged. This supplementary retrospective cohort is neither the complete hard-BD method nor a controlled hard-BD ablation.

\subsection{Staged complementary traction mismatches}
\begin{table}[!htbp]
\centering\small
\caption{Sequential stages for BDF development seed 20260911. Entries are nondimensional vector RMS in physical units; stages are not independent replicates.}
\label{s-tab:traction-history}
\begin{adjustbox}{max width=\linewidth}
\begin{tabular}{@{}lrrrr@{}}\toprule
Stage & SB native & SB recon. & BD native & BD recon.\\\midrule
1 & 0.040097 & 0.718156 & 0.081237 & 0.776907\\
2 & 0.286238 & 0.014755 & 0.235175 & 0.029116\\
3 & 0.019938 & 0.007286 & 0.030955 & 0.022864\\\bottomrule
\end{tabular}
\end{adjustbox}
\end{table}
Against a common diagnostic line of 0.1, stage 1 has small native but large reconstructed residuals, stage 2 the reverse, and stage 3 controls both. Stage 2 strengthens constitutive weights and adds reconstructed traction; stage 3 restores physical native-traction scaling. Each continuation also adds 300 Adam and 1000 L-BFGS iterations. This demonstrates complementary diagnostic utility, without isolating effects of extra training, weights or traction terms. The archived stage-2 pass flag refers only to its then-active reconstructed criterion and is retained unchanged.

\subsection{Additional adverse momentum diagnostics}
\begin{table}[H]\centering\normalsize\renewcommand{\arraystretch}{1.18}\setlength{\tabcolsep}{4pt}
\caption{BDF: number of seeds with increases after correction.}
\begin{adjustbox}{max width=\linewidth}
\begin{tabular}{@{}lcc@{}}\toprule
Metric & Main & Fine\\\midrule
B momentum grid RMS & 4/4 & 4/4\\
B momentum grid maximum component & 4/4 & 4/4\\
B momentum separate-evaluator RMS & 4/4 & --\\
Reconstructed B grid RMS & 4/4 & 4/4\\
Reconstructed B grid maximum component & 4/4 & 4/4\\
\bottomrule\end{tabular}
\end{adjustbox}\end{table}\section{Component pairs and traction diagnostics}
\subsection{Frozen settings and integrity}
The first paired batch contains six BD runs at $K_B=10^{-6}$ and six nonlinear-trace runs below. A separately prespecified cross-permeability extension adds 12 BD runs at $K_B=10^{-2},10^{-4}$. The main BD comparison uses 18 checkpoints, three pairs per coefficient with the same three seeds, not nine independent replicates.  All final-budget checkpoints are included, with matched initialization and samples and no seed selection, post-validation tuning or pressure shifts. Both batches use the same CPU float64 setup, two threads per process, at most two concurrent processes, 400 Adam updates plus two 300-iteration L-BFGS blocks with history reset, sampling and evaluation grids. Timings from different dates are not a cross-date speed ranking.

The coefficient-scan group uses seeds 3101, 3127, 3163, 3181 and 3203; the independent hard-BD group uses 4101, 4127, 4163, 4181 and 4203, all at nominal budgets of 400+300+300. Three nonpassing extreme soft-BD complete-route runs exceed the Stokes-pressure gate. These distinct-seed groups remain separate from the cross-permeability pairs.

\begin{table}[!htbp]
\centering\small
\caption{MMS2 coefficient stress test at fixed $K_D=10^{-2}$. Each group contains five seeds; passes require all 14 acceptance criteria on the main grid. The hard-BD route also passes 5/5 on the fine grid. Coefficient-scan and hard-BD routes use different seeds.}
\label{s-tab:extreme-comparison}
\begin{adjustbox}{max width=\linewidth}
\begin{tabular}{@{}lccc@{}}
\toprule
Configuration & $K_B$ & Passes & Mean $p_S$ $L^2$ (\%)\\
\midrule
Original hard trace & \(10^{-2}\) & 5/5 & 7.231\\
Original hard trace & \(10^{-4}\) & 2/5 & 11.663\\
Original hard trace & \(10^{-6}\) & 2/5 & 11.612\\
No pressure auxiliary & \(10^{-6}\) & 0/5 & 17.904\\
Hard BD traction & \(10^{-6}\) & 5/5 & 8.769\\
\bottomrule
\end{tabular}
\end{adjustbox}
\end{table}

\begin{table}[!htbp]
\centering\small
\caption{Hard BD traction over five independent validation seeds. All errors are percentages; standard deviations use $n-1$. The full $H^1$ norm includes field values and all first derivatives and is not part of the joint acceptance rule.}
\label{s-tab:extreme-fields}
\begin{adjustbox}{max width=\linewidth}
\begin{tabular}{@{}lccc@{}}
\toprule
Field & Mean $L^2$ $\pm$ s.d. & Worst $L^2$ & Mean $H^1$\\
\midrule
Stokes velocity & \(0.043\pm0.008\) & 0.053 & 0.157\\
Stokes pressure & \(8.769\pm0.455\) & 9.544 & 5.256\\
Brinkman velocity & \(0.287\pm0.065\) & 0.382 & 0.752\\
Brinkman pressure & \(5.530\pm0.343\) & 6.112 & 24.982\\
Darcy flux & \(1.620\pm0.090\) & 1.729 & 7.228\\
Darcy pressure & \(0.058\pm0.026\) & 0.103 & 0.157\\
\bottomrule
\end{tabular}
\end{adjustbox}
\end{table}
\begin{table}[!htbp]\centering\small\setlength{\tabcolsep}{4pt}
\caption{Uncorrected paired hard/soft BD results, $n=3$ per method and coefficient. Continuous metrics are mean $\pm$ sample standard deviation; main-grid pressure errors are percentages and $H^1$ is the full norm. Passes require all 14 acceptance criteria on both grids. Times are concurrent CPU measurements within each group; $10^{-6}$ is from the earlier batch.}
\label{s-tab:paired-bd}
\begin{adjustbox}{max width=\linewidth}
\begin{tabular}{@{}llccccc@{}}\toprule
$K_B$ & Traction & $p_S$: $L^2$ & $p_B$: $L^2$ & $p_B$: $H^1$ & Time (s) & Passes\\\midrule
\(10^{-2}\) & Hard & \(7.36\pm0.56\) & \(4.80\pm0.40\) & \(20.06\pm1.95\) & \(275.8\pm11.5\) & 3/3\\
\(10^{-2}\) & Soft & \(7.93\pm0.73\) & \(4.61\pm0.35\) & \(23.27\pm2.55\) & \(183.5\pm8.9\) & 3/3\\
\(10^{-4}\) & Hard & \(9.16\pm0.49\) & \(5.80\pm0.13\) & \(26.03\pm1.82\) & \(263.3\pm11.7\) & 3/3\\
\(10^{-4}\) & Soft & \(12.64\pm1.71\) & \(7.97\pm1.10\) & \(37.99\pm5.23\) & \(178.5\pm6.5\) & 0/3\\
\(10^{-6}\) & Hard & \(9.13\pm0.83\) & \(5.82\pm0.13\) & \(26.26\pm2.67\) & \(246.2\pm3.5\) & 3/3\\
\(10^{-6}\) & Soft & \(12.60\pm1.60\) & \(7.93\pm1.09\) & \(37.91\pm4.89\) & \(165.8\pm3.8\) & 0/3\\
\bottomrule\end{tabular}
\end{adjustbox}\end{table}
\subsection{Cross-permeability paired diagnostics}

Preflight equation and hard-constraint checks for the two higher permeabilities have maximum defect $7.11\times10^{-15}$. Initialization and samples are matched within pairs and across coefficients. The maximum main/fine field $L^2$ change over all 18 checkpoints is 0.000913 percentage points. Per-seed pressure results are reported once in Supplement S8; complete six-field, eight traction-component and conservation records accompany the manuscript.

\begin{table}[!htbp]\centering\small\setlength{\tabcolsep}{4pt}
\caption{Paired diagnostic means and sample standard deviations. BD reconstructed normal RMS is a single nondimensional component; global mass defect is a percentage.}
\label{s-tab:cross-traction-mass}
\begin{adjustbox}{max width=\linewidth}
\begin{tabular}{@{}llcc@{}}\toprule
$K_B$ & Traction & BD normal RMS & Mass defect (\%)\\\midrule
\(10^{-2}\) & Hard & \(0.01149\pm0.00213\) & \(0.03241\pm0.00109\)\\
\(10^{-2}\) & Soft & \(0.00870\pm0.00277\) & \(0.04207\pm0.00786\)\\
\(10^{-4}\) & Hard & \(0.01248\pm0.00168\) & \(0.01305\pm0.00323\)\\
\(10^{-4}\) & Soft & \(0.00856\pm0.00234\) & \(0.01473\pm0.00793\)\\
\(10^{-6}\) & Hard & \(0.01271\pm0.00205\) & \(0.01080\pm0.00627\)\\
\(10^{-6}\) & Soft & \(0.00857\pm0.00267\) & \(0.01871\pm0.00412\)\\
\bottomrule\end{tabular}
\end{adjustbox}\end{table}
Hard BD has lower full Brinkman-pressure $H^1$ error in every pair at each coefficient. At $10^{-2}$, Stokes-pressure $L^2$ is lower in 2/3 pairs and Brinkman-pressure $L^2$ in 1/3; the hard/soft means are 7.36\%/7.93\% and 4.80\%/4.61\%, respectively. Mean BD reconstructed normal RMS is larger for hard BD at all three coefficients; training time increases by about 50.3\%, 47.6\% and 48.5\%. All soft-BD failures at $10^{-4}$ and $10^{-6}$ concern Stokes-pressure $L^2$. These records retain the trade-offs across metrics.

\paragraph{Independent five-seed diagnostics}
The following quantities concern the five independent hard-BD validation seeds 4101, 4127, 4163, 4181 and 4203 in \cref{s-tab:extreme-fields}; they are separate from the three-seed cross-permeability pairs above. 
The component-averaged reconstructed traction RMS is 0.011953 at SB and 0.009824 at BD (vector RMS is $\sqrt{2}$ times larger: 0.016903 and 0.013893). The mean global exterior mass defect is 0.01221\%.
The mass defect integrates all exterior edges, including Darcy side flux,
and is normalized by total absolute boundary flux and source integrals.
Brinkman momentum has mean raw component RMS 749.902, versus 0.000749901
after division by \(1+K_B^{-1}\); its physical-term relative residual is
0.00143595 (denominator: the sum of the RMS stress-divergence, drag, and
forcing terms). The smaller scaled number is not a raw-equation accuracy
claim. Mean Brinkman-pressure full \(H^1\) error remains 24.982\%.
Stokes-pressure mean bias accounts for 99.13\% of its squared \(L^2\)
error on average; this is diagnostic only, without correcting the pressure.
Mean training time is 203.8 s under concurrent CPU execution, which does not
establish a speed advantage. The scope of these results is summarized in the conclusions.

\begin{figure}[!htbp]\centering
\includegraphics[width=\linewidth]{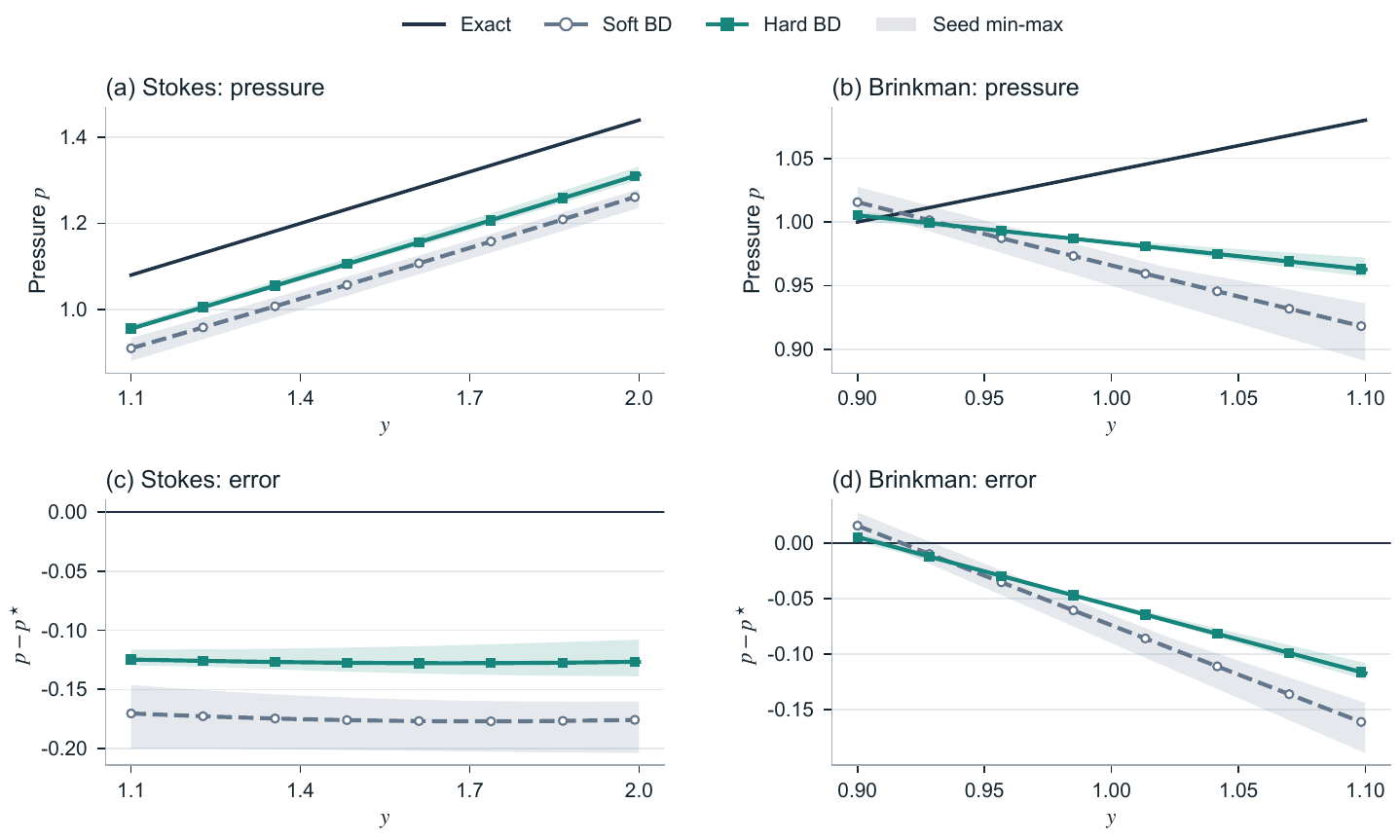}
\caption{Extreme-MMS2 paired centerline pressures and signed errors, $x=0.5$. Black: exact; colored: means over seeds 4301, 4327 and 4363, with min--max bands. All 121 centerline points per region use the original main-grid outputs, without retraining or pressure shifts.}
\label{s-fig:archived-pressure-profiles}\end{figure}
\subsection{Nonlinear Darcy pressure trace}
Set $K_B=K_D=0.01$ and add to the MMS2 pressure in each region
\[
 H(x,y)=0.25\sin(\pi x)\cosh\bigl(\pi(y-0.9)\bigr).
\]
Keep the upper velocities unchanged, add $-(K_D/\mu)\nabla H$ to the Darcy flux, and update forcing and exterior data accordingly. Since $\Delta H=0$ and $\partial_yH=0$ at BD, mass equations and interface loads remain compatible. Preflight PDE, interface and corner checks have maximum defect $7.11\times10^{-15}$.

Both variants use hard BD traction. Only the trainable Darcy pressure-trace correction is multiplied by $c=1$ (split) or $c=\mu/K_D=100$ (amplified); the interior bubble scaling is unchanged. Network capacity, function class, paired raw weights and samples match, but initial physical trace amplitudes differ. This does not isolate optimization from initialization effects. Seed 20260910 is developmental; independent seeds 5101, 5127 and 5163 are run after freezing.

\begin{table}[!htbp]
\centering\small
\caption{Three-seed means for the nonlinear BD pressure trace. Relative trace errors use 1601 interface points and no pressure shift; full $H^1$ includes trace values and tangential derivatives. The 14 gates are an inherited diagnostic screen, not application-validated accuracy requirements for this test.}
\label{s-tab:nonlinear-trace}
\begin{adjustbox}{max width=\linewidth}
\begin{tabular}{@{}lrrrr@{}}\toprule
Variant & Both-grid passes & $L^2$ (\%) & $H^1$ (\%) & Time (s)\\\midrule
split & 0/3 & 22.6331 & 40.0937 & 276.05\\
amplified & 0/3 & 18.6877 & 32.7651 & 274.85\\\bottomrule
\end{tabular}
\end{adjustbox}
\end{table}
Amplification reduces mean error, but neither variant passes the joint screen. These results do not support accurate nonlinear pressure-trace learning at this budget or general superiority of scale separation. The analytical endpoint-linear interpolant has trace $L^2$ error 13.6269\%; it is a mathematical reference, not a trained control. Developmental errors are 25.54\% and 19.58\%, respectively, and are excluded from formal means.

\section{Pressure-mode control and interface-component validation}
\subsection{Mean flow from exterior mass balance}
Let $a=y_{BD}$, $b=y_{SB}$, $t=y_T$, $h=b-a$, $L=x_1-x_0$, and $D_g(y)=g_x(x_1,y)-g_x(x_0,y)$ denote the prescribed side-velocity difference in the relevant region. For divergence-free upper flow with SB normal continuity, $U(y)=\int_{x_0}^{x_1}u_y(x,y)\,dx$ satisfies $U'(y)=-D_g(y)$. Integrating from the prescribed top velocity gives
\begin{equation}
\bar u_{B,y}^{\mathrm{bc}}=
\frac{1}{L}\int_{x_0}^{x_1}g_y(x,t)\,dx
+\frac{1}{L}\int_b^t D_g(s)\,ds
+\frac{1}{Lh}\int_a^b(s-a)D_g(s)\,ds.
\label{s-eq:boundary-mean-flow}
\end{equation}
This expression uses geometry and prescribed exterior velocities only. The divergence-free field $u=(xy,-y^2/2)$, with nonzero side-flux difference, independently checks the sign and layer weights: both analytical and numerical means are $-0.5016666667$ for this geometry. No interior velocity labels enter the formula.
\subsection{Relation between the physical moment and learned momentum}
The correction direction, its unit response and its traction invariants are established in the main-text pressure-correction section. One additional relation explains why the physical moment differs from the raw mean momentum of an approximate network:
\begin{equation}
M_B(z)-\langle(\R_B^m)_y\rangle_{\Omega_B}
=\frac{\mu}{K_B}\left(\bar u_{B,y}^{\mathrm{bc}}-\langle u_{B,y}\rangle_{\Omega_B}\right).
\label{s-eq:moment-discrepancy}
\end{equation}
For an exact incompressible solution, the two means agree and both momentum quantities vanish. Approximate velocities need not be exactly divergence-free; their discrepancy is amplified by drag. This explains why reducing the boundary-mass-consistent moment need not reduce pointwise momentum RMS.
\subsection{Development, freezing and independent validation}
Prior seed 4301 is used only for development and excluded from validation. Direct slope fitting from the raw momentum mean at 512 training points gives $c\approx-816.27$ and severely worsens pressure; using the same order-32 Gauss rule still gives about $-816.15$. The boundary-mass-consistent rule instead gives $c\approx0.606322$, agreeing at orders 16, 32 and 48. Failed development records are retained, not counted as successful validation.

Order-32 coefficient fitting, raw/corrected dual-grid evaluation and validation seeds 6101, 6127 and 6163 are frozen before validation. Two interface variants per seed give six CPU float64 runs, with two threads per process and at most two concurrent processes. Each uses 512 interior points per region, 192 exterior points per region and 192 interface points, 400 Adam updates (0.001), and two 300-iteration L-BFGS blocks (0.8, history 50, strong-Wolfe, reset between blocks). Coefficients are recorded before exact-pressure error evaluation. Order 48 independently checks the moment; main/fine grids are 161 by 121/321 by 241, with 801/1601 interface points. The original 14 acceptance criteria are unchanged, with kinematic and normal-pressure-derivative diagnostics reported additionally.
\subsection{Equal-parameter interface-release control}
Both variants retain the same regional/trace networks, auxiliaries, scaling, hard exterior data and hard native BD traction. With $\phi(x)=\xi(1-\xi)$, existing raw regional outputs define
\begin{align}
J_{SB}(x)&=\phi(x)\bigl(\widehat u_S(x,b)-\widehat u_B(x,b)\bigr),\\
J_{BD}(x)&=\phi(x)\bigl(\widehat u_{B,y}(x,a)-\widehat q_{D,y}(x,a)\bigr).
\label{s-eq:released-jumps}
\end{align}
SB traces use $T_{SB}+J_{SB}/2$ and $T_{SB}-J_{SB}/2$; BD normal traces add and subtract $J_{BD}/2$. Original lifting factors extend the changes, which vanish at endpoints; BD tangential velocity is unchanged. Pressure parameterization and residual-group coefficients are fixed. Hard kinematic residuals are identities, while the released variant uses the same groups as soft penalties. Both have 46,074 trainable parameters. Raw weights and samples match pairwise, though initial physical fields differ. This is a controlled release of the present method, not an asserted reproduction of another published architecture.

\begin{table}[H]\centering\small\setlength{\tabcolsep}{4pt}
\caption{Kinematic, conservation and raw-momentum means and sample standard deviations. Mass defect is a percentage; other entries are nondimensional RMS.}
\label{s-tab:control-invariants}
\begin{adjustbox}{max width=\linewidth}
\begin{tabular}{@{}lcccc@{}}\toprule
State & SB jump & BD jump & Mass (\%) & Raw B momentum\\\midrule
Hard/raw & \(0.00000\pm0.00000\) & \(0.00000\pm0.00000\) & \(0.00630\pm0.00080\) & \(728.56\pm139.43\)\\
Hard/corrected & \(0.00000\pm0.00000\) & \(0.00000\pm0.00000\) & \(0.00630\pm0.00080\) & \(728.62\pm139.51\)\\
Released/raw & \(0.00083\pm0.00010\) & \(0.00220\pm0.00048\) & \(0.01591\pm0.01357\) & \(949.15\pm198.47\)\\
Released/corrected & \(0.00083\pm0.00010\) & \(0.00220\pm0.00048\) & \(0.01591\pm0.01357\) & \(948.87\pm198.43\)\\
\bottomrule\end{tabular}
\end{adjustbox}\end{table}

\begin{table}[H]\centering\small\setlength{\tabcolsep}{4pt}
\caption{Cost and centerline slope. Training time is repeated for raw/corrected states of each checkpoint, not summed; the exact secant slope is 0.4.}
\label{s-tab:control-cost}
\begin{adjustbox}{max width=\linewidth}
\begin{tabular}{@{}lccc@{}}\toprule
State & Training (s) & Correction (s) & Slope\\\midrule
Hard/raw & \(247.14\pm0.91\) & \(0.0000\pm0.0000\) & \(-0.1833\pm0.1021\)\\
Hard/corrected & \(247.14\pm0.91\) & \(0.0545\pm0.0031\) & \(0.3774\pm0.0112\)\\
Released/raw & \(336.63\pm38.31\) & \(0.0000\pm0.0000\) & \(-0.2121\pm0.1099\)\\
Released/corrected & \(336.63\pm38.31\) & \(0.0633\pm0.0072\) & \(0.3719\pm0.0148\)\\
\bottomrule\end{tabular}
\end{adjustbox}\end{table}

\begin{table}[H]\centering\small\setlength{\tabcolsep}{4pt}
\caption{Pressure errors (percent) for all validation states; raw/corrected share each trained endpoint.}
\label{s-tab:control-seeds}
\begin{adjustbox}{max width=\linewidth}
\begin{tabular}{@{}llrrrr@{}}\toprule
Seed & State & $p_S$: $L^2$ & $p_B$: $L^2$ & $p_B$: $H^1$ & $\partial_y p_B$: $L^2$\\\midrule
6101 & hard/raw & 7.451 & 4.702 & 22.199 & 129.661\\
6101 & hard/corrected & 0.880 & 0.886 & 4.443 & 24.234\\
6101 & released/raw & 7.173 & 4.645 & 24.100 & 140.956\\
6101 & released/corrected & 1.212 & 1.209 & 4.329 & 23.369\\
6127 & hard/raw & 7.124 & 4.877 & 23.020 & 134.525\\
6127 & hard/corrected & 0.933 & 0.718 & 5.064 & 28.385\\
6127 & released/raw & 7.361 & 4.866 & 23.020 & 134.566\\
6127 & released/corrected & 0.967 & 0.718 & 4.998 & 28.149\\
6163 & hard/raw & 9.617 & 5.995 & 29.881 & 174.863\\
6163 & hard/corrected & 0.794 & 1.001 & 5.207 & 29.122\\
6163 & released/raw & 9.267 & 5.592 & 31.382 & 183.857\\
6163 & released/corrected & 1.257 & 1.716 & 5.374 & 29.818\\
\bottomrule\end{tabular}
\end{adjustbox}\end{table}
 The correction controls one low-order pressure direction using prescribed boundary data; it does not guarantee reduction of every raw momentum residual. The rule uses the stated steady, constant-coefficient setting and top/side velocity conditions; other geometries and mixed boundaries require a new derivation.

\begin{figure}[!htbp]\centering
\includegraphics[width=\linewidth]{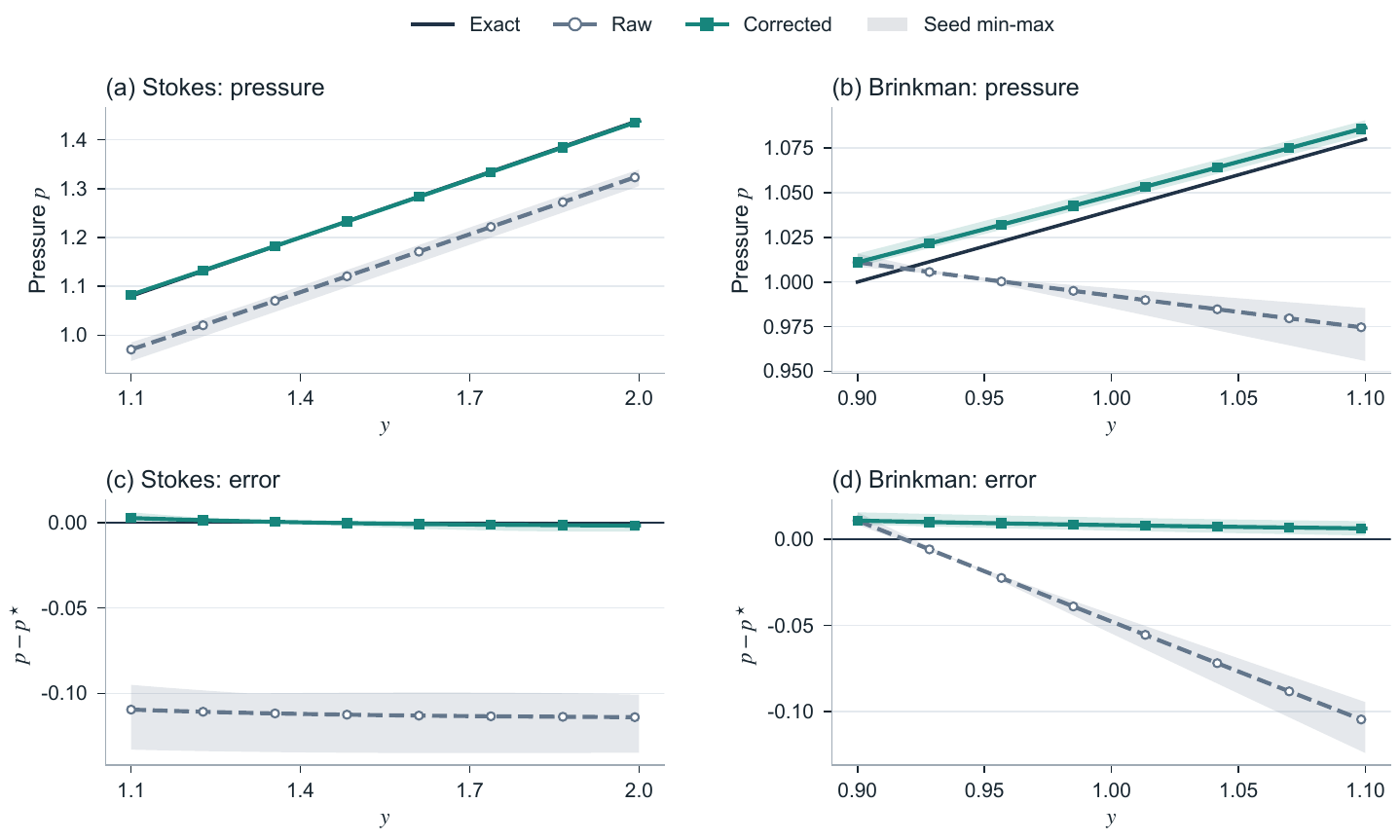}
\caption{Hard-shared interface-control profiles at $x=0.5$: exact pressures and raw/corrected means with min--max bands over seeds 6101, 6127 and 6163, pairing the same trained endpoints.}\label{s-fig:pressure-profiles}
\end{figure}

\section{Complete-method validation across MMS cases and permeabilities}\label{s-sec:coverage-supp}
We test MMS1 and MMS2 at three Brinkman permeabilities: $K_B=10^{-2}$, $10^{-4}$ and $10^{-6}$. Darcy permeability is fixed at $K_D=10^{-2}$. Each uses the complete hard-kinematic, hard-BD method with three seeds from \cref{s-tab:seed-registry}. The primary summary for these six combinations is given once in main-text Table~5. This section adds per-run records, derivative/momentum diagnostics and low-permeability grid checks. Spatial pressure behavior is shown in main-text Figures~6, 8 and 9.

Every run uses 512 interior and 192 exterior points per region and 192 per interface, 400 Adam updates at $10^{-3}$, then two 300-iteration L-BFGS blocks at rate 0.8 with history 50, strong-Wolfe search and a history reset. Order-32 Gauss integration fixes the correction before reference-error evaluation; order 48 checks it independently. The grids and 14-gate rule are specified in \cref{s-tab:evaluation-grids,s-tab:additional-gates}. In MMS2 the known boundary/forcing term in the correction reduces analytically to 1.6 and is checked by quadrature at all three permeabilities; it is not fitted to pressure errors.
\begin{table}[!htbp]\centering\footnotesize\setlength{\tabcolsep}{3pt}
\caption{Complete-method validation for both MMS cases at all three permeabilities: main-grid raw $\to$ corrected pressure errors (percent) for all 18 checkpoints; the last column records both-grid acceptance of the two states.}\label{s-tab:coverage-seeds-prospective}
\begin{adjustbox}{max width=\linewidth}
\begin{tabular}{@{}lllcccc@{}}\toprule
Case, $K_B$ & BD & Seed & $p_S$: $L^2$ & $p_B$: $L^2$ & $p_B$: $H^1$ & Pass\\\midrule
MMS1, $10^{-2}$ & H & 6201 & $3.180\to0.857$ & $4.053\to1.976$ & $15.544\to7.803$ & Y$\to$Y\\
MMS1, $10^{-2}$ & H & 6227 & $6.472\to1.459$ & $6.121\to2.462$ & $20.477\to4.090$ & Y$\to$Y\\
MMS1, $10^{-2}$ & H & 6263 & $1.758\to1.442$ & $3.494\to2.852$ & $11.749\to10.716$ & Y$\to$Y\\
MMS1, $10^{-4}$ & H & 6201 & $3.534\to1.546$ & $4.568\to2.014$ & $17.405\to8.093$ & Y$\to$Y\\
MMS1, $10^{-4}$ & H & 6227 & $7.166\to1.549$ & $7.719\to3.335$ & $26.022\to5.548$ & Y$\to$Y\\
MMS1, $10^{-4}$ & H & 6263 & $1.998\to1.389$ & $3.662\to2.394$ & $12.142\to10.453$ & Y$\to$Y\\
MMS1, $10^{-6}$ & H & 6201 & $3.583\to1.290$ & $4.775\to2.291$ & $16.637\to8.425$ & Y$\to$Y\\
MMS1, $10^{-6}$ & H & 6227 & $6.852\to1.378$ & $7.379\to2.788$ & $26.963\to5.077$ & Y$\to$Y\\
MMS1, $10^{-6}$ & H & 6263 & $2.277\to1.404$ & $3.775\to2.383$ & $12.999\to10.253$ & Y$\to$Y\\
MMS2, $10^{-2}$ & H & 6201 & $7.688\to1.368$ & $4.721\to1.020$ & $22.832\to5.595$ & Y$\to$Y\\
MMS2, $10^{-2}$ & H & 6227 & $5.912\to1.069$ & $3.829\to0.787$ & $17.717\to5.589$ & Y$\to$Y\\
MMS2, $10^{-2}$ & H & 6263 & $7.482\to1.876$ & $4.938\to1.118$ & $19.951\to4.674$ & Y$\to$Y\\
MMS2, $10^{-4}$ & H & 6201 & $10.233\to0.812$ & $6.504\to1.165$ & $31.796\to4.919$ & N$\to$Y\\
MMS2, $10^{-4}$ & H & 6227 & $6.727\to0.795$ & $4.337\to0.747$ & $20.186\to5.312$ & Y$\to$Y\\
MMS2, $10^{-4}$ & H & 6263 & $8.815\to1.212$ & $5.880\to0.708$ & $24.292\to4.077$ & Y$\to$Y\\
MMS2, $10^{-6}$ & H & 6201 & $10.154\to0.766$ & $6.389\to1.293$ & $31.317\to5.090$ & N$\to$Y\\
MMS2, $10^{-6}$ & H & 6227 & $6.689\to0.814$ & $4.268\to0.765$ & $20.039\to5.346$ & Y$\to$Y\\
MMS2, $10^{-6}$ & H & 6263 & $8.545\to1.202$ & $5.888\to0.864$ & $23.592\to4.001$ & Y$\to$Y\\
\bottomrule\end{tabular}
\end{adjustbox}\end{table}

Per-run normal-derivative checks can be reproduced from \nolinkurl{data/mms_complete_method_results.json}, which is available from the authors upon reasonable request. In its \texttt{runs} records, the 12 checkpoints with \texttt{cohort=prospective} have a smaller \texttt{brinkman\_\allowbreak p\_\allowbreak y\_\allowbreak l2\_\allowbreak relative} after correction. For the six low-permeability MMS1 checkpoints, the \nolinkurl{MMS1_low_permeability_metrics} records show a smaller \texttt{brinkman\_\allowbreak p\_\allowbreak h1\_\allowbreak semi\_\allowbreak abs\_\allowbreak error} on both grids in every run. Since $\partial_x p_B$ is unchanged by the correction, this decrease in the gradient-error seminorm implies a decrease in the $L^2$ error of $\partial_y p_B$. The \nolinkurl{MMS1_low_permeability_runs} and \nolinkurl{MMS1_low_permeability_gates} records in the same JSON file record all 14 criteria passing for each raw/corrected state on each grid.

\begin{table}[!htbp]\centering\small\setlength{\tabcolsep}{4pt}
\caption{Complete-method diagnostic means. Derivative errors are percentages; raw momentum is component-averaged RMS. Training and correction costs are seconds; training is counted once per endpoint.}\label{s-tab:coverage-cost-gradient}
\begin{adjustbox}{max width=\linewidth}
\begin{tabular}{@{}lcccc@{}}\toprule
Case, $K_B$ & $\partial_y p_B$: raw/corr. & B mom.: raw/corr. & Train (s) & Corr. (s)\\\midrule
MMS1, $10^{-2}$ & 24.37/9.62 & 0.26/0.24 & 239.92 & 0.0506\\
MMS2, $10^{-2}$ & 117.66/29.49 & 0.23/0.17 & 250.97 & 0.0534\\
MMS2, $10^{-4}$ & 148.52/26.02 & 6.99/7.15 & 248.77 & 0.0581\\
MMS2, $10^{-6}$ & 145.88/26.07 & 759.43/759.61 & 249.53 & 0.0615\\
\bottomrule\end{tabular}
\end{adjustbox}\end{table}

\clearpage
\subsection{Low-permeability MMS1: fine-grid and momentum diagnostics}
The two additional low-permeability MMS1 cases use the same complete-method protocol as the other four configurations. Primary pressure means are in main-text Table~5; the tables below add fine-grid pressure checks and main/fine momentum diagnostics. All raw/corrected states pass the 14-gate rule, and all six runs pass the order-48 and invariance checks.

\begin{table}[H]\centering\normalsize\renewcommand{\arraystretch}{1.18}\setlength{\tabcolsep}{4pt}
\caption{MMS1, $K_B=10^{-4}$. Fine-grid pressure and main/fine momentum: mean $\pm$ sample SD; final column: decreases/increases/unchanged.}
\begin{adjustbox}{max width=\linewidth}
\begin{tabular}{@{}lcccc@{}}\toprule
Grid & Metric & Raw & Corrected & $\downarrow/\uparrow/=$\\\midrule
main & B momentum RMS & \(12.324\pm2.469\) & \(12.303\pm2.482\) & 3/0/0\\
main & B momentum maximum component & \(29.502\pm3.706\) & \(29.612\pm3.894\) & 0/1/2\\
fine & $p_S$: $L^2$ (\%) & \(4.232\pm2.654\) & \(1.494\pm0.091\) & 3/0/0\\
fine & $p_B$: $L^2$ (\%) & \(5.316\pm2.130\) & \(2.581\pm0.680\) & 3/0/0\\
fine & $p_B$: $H^1$ (\%) & \(18.523\pm7.007\) & \(8.031\pm2.453\) & 3/0/0\\
fine & B momentum RMS & \(12.324\pm2.469\) & \(12.303\pm2.482\) & 3/0/0\\
fine & B momentum maximum component & \(29.503\pm3.705\) & \(29.613\pm3.893\) & 0/1/2\\
\bottomrule\end{tabular}
\end{adjustbox}\end{table}

\begin{table}[H]\centering\normalsize\renewcommand{\arraystretch}{1.18}\setlength{\tabcolsep}{4pt}
\caption{MMS1, $K_B=10^{-6}$. Fine-grid pressure and main/fine momentum: mean $\pm$ sample SD; final column: decreases/increases/unchanged.}
\begin{adjustbox}{max width=\linewidth}
\begin{tabular}{@{}lcccc@{}}\toprule
Grid & Metric & Raw & Corrected & $\downarrow/\uparrow/=$\\\midrule
main & B momentum RMS & \(1303.942\pm319.294\) & \(1303.916\pm319.279\) & 3/0/0\\
main & B momentum maximum component & \(3133.675\pm583.484\) & \(3133.571\pm583.315\) & 1/0/2\\
fine & $p_S$: $L^2$ (\%) & \(4.237\pm2.356\) & \(1.357\pm0.060\) & 3/0/0\\
fine & $p_B$: $L^2$ (\%) & \(5.310\pm1.861\) & \(2.487\pm0.264\) & 3/0/0\\
fine & $p_B$: $H^1$ (\%) & \(18.866\pm7.245\) & \(7.918\pm2.625\) & 3/0/0\\
fine & B momentum RMS & \(1303.941\pm319.292\) & \(1303.915\pm319.277\) & 3/0/0\\
fine & B momentum maximum component & \(3133.750\pm583.473\) & \(3133.646\pm583.304\) & 1/0/2\\
\bottomrule\end{tabular}
\end{adjustbox}\end{table}

\section{Pressure-correction controls for benchmarks and BD pairs}\label{s-sec:retrospective-supp}
The ten kinematic-hard-trace MMS benchmark checkpoints and eighteen hard/soft-BD paired checkpoints are evaluated before and after the same boundary-mass-consistent correction. These 28 controls are separate from complete-method validation. The tables report per-checkpoint pressures and joint acceptance of each state; aggregate pressure conclusions are discussed in the main text. Here all records use the extended MMS grids and component 14-gate diagnostic in \cref{s-tab:evaluation-grids,s-tab:additional-gates}. Original benchmark acceptance remains the main-text benchmark rule on its original grid. Correction coefficients are determined before reference-error evaluation, and pressure means are not removed.
\begin{table}[!htbp]\centering\footnotesize\setlength{\tabcolsep}{3pt}
\caption{Retrospective MMS: raw $\to$ corrected pressure errors (percent) for every checkpoint; the last column records both-grid acceptance of the two states.}\label{s-tab:coverage-seeds-retrospective-mms}
\begin{adjustbox}{max width=\linewidth}
\begin{tabular}{@{}lllcccc@{}}\toprule
Case, $K_B$ & BD & Seed & $p_S$: $L^2$ & $p_B$: $L^2$ & $p_B$: $H^1$ & Pass\\\midrule
MMS1, $10^{-2}$ & S & 20260901 & $1.884\to2.026$ & $3.011\to3.135$ & $5.619\to5.678$ & Y$\to$Y\\
MMS1, $10^{-2}$ & S & 20260902 & $2.677\to3.718$ & $4.299\to6.523$ & $22.448\to13.282$ & Y$\to$Y\\
MMS1, $10^{-2}$ & S & 20260903 & $4.274\to2.982$ & $4.343\to4.242$ & $25.020\to11.852$ & Y$\to$Y\\
MMS1, $10^{-2}$ & S & 20260904 & $7.048\to3.279$ & $6.446\to5.852$ & $40.713\to13.908$ & Y$\to$Y\\
MMS1, $10^{-2}$ & S & 20260905 & $2.639\to2.379$ & $2.952\to4.692$ & $18.725\to10.868$ & Y$\to$Y\\
MMS2, $10^{-2}$ & S & 20260901 & $5.965\to0.536$ & $3.535\to0.946$ & $18.994\to4.941$ & Y$\to$Y\\
MMS2, $10^{-2}$ & S & 20260902 & $8.595\to1.021$ & $5.187\to1.003$ & $26.966\to6.005$ & Y$\to$Y\\
MMS2, $10^{-2}$ & S & 20260903 & $7.077\to1.371$ & $3.887\to1.262$ & $22.101\to6.503$ & Y$\to$Y\\
MMS2, $10^{-2}$ & S & 20260904 & $8.387\to1.874$ & $5.009\to0.926$ & $24.893\to7.568$ & Y$\to$Y\\
MMS2, $10^{-2}$ & S & 20260905 & $6.883\to0.974$ & $4.035\to0.893$ & $21.484\to6.286$ & Y$\to$Y\\
\bottomrule\end{tabular}
\end{adjustbox}\end{table}
\begin{table}[!htbp]\centering\footnotesize\setlength{\tabcolsep}{3pt}
\caption{Retrospective pairs: raw $\to$ corrected pressure errors (percent) for every checkpoint; the last column records both-grid acceptance of the two states.}\label{s-tab:coverage-seeds-retrospective-pairs}
\begin{adjustbox}{max width=\linewidth}
\begin{tabular}{@{}lllcccc@{}}\toprule
Case, $K_B$ & BD & Seed & $p_S$: $L^2$ & $p_B$: $L^2$ & $p_B$: $H^1$ & Pass\\\midrule
MMS2, $10^{-2}$ & H & 4301 & $7.993\to1.938$ & $5.084\to1.179$ & $22.165\to6.941$ & Y$\to$Y\\
MMS2, $10^{-2}$ & S & 4301 & $8.357\to1.292$ & $4.741\to1.225$ & $25.374\to7.853$ & Y$\to$Y\\
MMS2, $10^{-2}$ & H & 4327 & $6.953\to1.504$ & $4.346\to0.693$ & $19.702\to4.651$ & Y$\to$Y\\
MMS2, $10^{-2}$ & S & 4327 & $8.338\to1.390$ & $4.870\to0.730$ & $24.002\to5.800$ & Y$\to$Y\\
MMS2, $10^{-2}$ & H & 4363 & $7.130\to1.953$ & $4.981\to1.548$ & $18.311\to4.285$ & Y$\to$Y\\
MMS2, $10^{-2}$ & S & 4363 & $7.087\to1.463$ & $4.205\to0.644$ & $20.436\to5.186$ & Y$\to$Y\\
MMS2, $10^{-4}$ & H & 4301 & $9.594\to1.274$ & $5.951\to0.885$ & $27.250\to6.534$ & Y$\to$Y\\
MMS2, $10^{-4}$ & S & 4301 & $12.780\to1.022$ & $7.861\to1.533$ & $39.678\to7.136$ & N$\to$Y\\
MMS2, $10^{-4}$ & H & 4327 & $9.268\to1.065$ & $5.716\to0.639$ & $26.908\to4.210$ & Y$\to$Y\\
MMS2, $10^{-4}$ & S & 4327 & $14.281\to0.824$ & $9.119\to0.804$ & $42.169\to5.132$ & N$\to$Y\\
MMS2, $10^{-4}$ & H & 4363 & $8.624\to1.286$ & $5.741\to1.041$ & $23.942\to4.011$ & Y$\to$Y\\
MMS2, $10^{-4}$ & S & 4363 & $10.867\to0.956$ & $6.935\to0.636$ & $32.128\to4.407$ & N$\to$Y\\
MMS2, $10^{-6}$ & H & 4301 & $9.700\to1.413$ & $5.897\to0.915$ & $27.561\to6.661$ & Y$\to$Y\\
MMS2, $10^{-6}$ & S & 4301 & $12.146\to1.243$ & $7.600\to1.684$ & $38.967\to7.295$ & N$\to$Y\\
MMS2, $10^{-6}$ & H & 4327 & $9.515\to1.046$ & $5.901\to0.675$ & $28.029\to4.246$ & Y$\to$Y\\
MMS2, $10^{-6}$ & S & 4327 & $14.379\to0.974$ & $9.147\to0.776$ & $42.178\to5.221$ & N$\to$Y\\
MMS2, $10^{-6}$ & H & 4363 & $8.171\to1.144$ & $5.670\to1.004$ & $23.193\to4.008$ & Y$\to$Y\\
MMS2, $10^{-6}$ & S & 4363 & $11.274\to1.205$ & $7.043\to0.593$ & $32.572\to4.505$ & N$\to$Y\\
\bottomrule\end{tabular}
\end{adjustbox}\end{table}

\end{document}